\documentclass[11pt]{amsart}

\usepackage[all]{xy}
\usepackage{color, pstricks}
\usepackage{mathrsfs}
\usepackage{calligra}
\usepackage{amsmath,amsthm, amssymb, amsgen,amsxtra,amsfonts, amsbsy} 
\usepackage[all]{xy}
\usepackage{verbatim}
\usepackage{url}
\usepackage{hyperref}

\DeclareMathOperator{\cha}{char}

\DeclareMathOperator{\Gal}{Gal}

\DeclareMathOperator{\Res}{Res}

\DeclareMathOperator{\diag}{diag}

\DeclareMathOperator{\shHom}{\mathscr{H}\text{\kern -3pt {\calligra\large om}}\,}
\DeclareMathOperator{\shEnd}{\mathscr{E}\text{\kern -3pt {\calligra\large nd}}\,}

\begin{document}
%
%
%
\theoremstyle{definition}
\newtheorem{Definition}{Definition}[section]
\newtheorem*{Definitionx}{Definition}
\newtheorem{Convention}[Definition]{Convention}
\newtheorem*{Conventionsx}{Conventions}
\newtheorem{Construction}{Construction}[section]
\newtheorem{Example}[Definition]{Example}
\newtheorem{Examples}[Definition]{Examples}
\newtheorem{Remark}[Definition]{Remark}
\newtheorem{Setup}[Definition]{Setup}
\newtheorem*{Remarkx}{Remark}
\newtheorem{Remarks}[Definition]{Remarks}
\newtheorem{Caution}[Definition]{Caution}
\newtheorem{Conjecture}[Definition]{Conjecture}
\newtheorem*{Conjecturex}{Conjecture}
\newtheorem{Question}[Definition]{Question}
\newtheorem{Questions}[Definition]{Questions}
\newtheorem*{Acknowledgements}{Acknowledgements}
\newtheorem*{Organization}{Organization}
\newtheorem*{Disclaimer}{Disclaimer}
\theoremstyle{plain}
\newtheorem{Theorem}[Definition]{Theorem}
\newtheorem*{Theoremx}{Theorem}
\newtheorem{Theoremy}{Theorem}
\newtheorem{Proposition}[Definition]{Proposition}
\newtheorem*{Propositionx}{Proposition}
\newtheorem{Lemma}[Definition]{Lemma}
\newtheorem{Corollary}[Definition]{Corollary}
\newtheorem*{Corollaryx}{Corollary}
\newtheorem{Fact}[Definition]{Fact}
\newtheorem{Facts}[Definition]{Facts}
\newtheoremstyle{voiditstyle}{3pt}{3pt}{\itshape}{\parindent}%
{\bfseries}{.}{ }{\thmnote{#3}}%
\theoremstyle{voiditstyle}
\newtheorem*{VoidItalic}{}
\newtheoremstyle{voidromstyle}{3pt}{3pt}{\rm}{\parindent}%
{\bfseries}{.}{ }{\thmnote{#3}}%
\theoremstyle{voidromstyle}
\newtheorem*{VoidRoman}{}

%
\newcommand{\prf}{\par\noindent{\sc Proof.}\quad}
\newcommand{\blowup}{\rule[-3mm]{0mm}{0mm}}
\newcommand{\cal}{\mathcal}
\newcommand{\Aff}{{\mathds{A}}}
\newcommand{\BB}{{\mathds{B}}}
\newcommand{\CC}{{\mathds{C}}}
\newcommand{\EE}{{\mathds{E}}}
\newcommand{\FF}{{\mathds{F}}}
\newcommand{\GG}{{\mathds{G}}}
\newcommand{\HH}{{\mathds{H}}}
\newcommand{\NN}{{\mathds{N}}}
\newcommand{\ZZ}{{\mathds{Z}}}
\newcommand{\PP}{{\mathds{P}}}
\newcommand{\QQ}{{\mathds{Q}}}
\newcommand{\RR}{{\mathds{R}}}
\newcommand{\Sph}{{\mathds{S}}}
\newcommand{\TT}{{\mathds{T}}}
\newcommand{\Liea}{{\mathfrak a}}
\newcommand{\Lieb}{{\mathfrak b}}
\newcommand{\Lieg}{{\mathfrak g}}
\newcommand{\Liem}{{\mathfrak m}}
\newcommand{\ideala}{{\mathfrak a}}
\newcommand{\idealb}{{\mathfrak b}}
\newcommand{\idealg}{{\mathfrak g}}
\newcommand{\idealj}{{\mathfrak j}}
\newcommand{\idealm}{{\mathfrak m}}
\newcommand{\idealn}{{\mathfrak n}}
\newcommand{\idealp}{{\mathfrak p}}
\newcommand{\idealq}{{\mathfrak q}}
\newcommand{\idealI}{{\cal I}}
\newcommand{\lin}{\sim}
\newcommand{\num}{\equiv}
\newcommand{\dual}{\ast}
\newcommand{\iso}{\cong}
\newcommand{\homeo}{\approx}
\newcommand{\mathds}[1]{{\mathbb #1}}
\newcommand{\mm}{{\mathfrak m}}
\newcommand{\pp}{{\mathfrak p}}
\newcommand{\qq}{{\mathfrak q}}
\newcommand{\rr}{{\mathfrak r}}
\newcommand{\pP}{{\mathfrak P}}
\newcommand{\qQ}{{\mathfrak Q}}
\newcommand{\rR}{{\mathfrak R}}
%
%
\newcommand{\OO}{{\cal O}}
\newcommand{\calA}{{\cal A}}
\newcommand{\calD}{{\cal D}}
\newcommand{\calM}{{\cal M}}
\newcommand{\calO}{{\cal O}}
\newcommand{\calP}{{\cal P}}
\newcommand{\calT}{{\cal T}}
\newcommand{\calU}{{\cal U}}
\newcommand{\numero}{{n$^{\rm o}\:$}}
\newcommand{\mf}[1]{\mathfrak{#1}}
\newcommand{\mc}[1]{\mathcal{#1}}
\newcommand{\into}{{\hookrightarrow}}
\newcommand{\onto}{{\twoheadrightarrow}}
\newcommand{\Spec}{{\rm Spec}\:}
\newcommand{\BigSpec}{{\rm\bf Spec}\:}
\newcommand{\Spf}{{\rm Spf}\:}
\newcommand{\Proj}{{\rm Proj}\:}
\newcommand{\Pic}{{\rm Pic }}
\newcommand{\Picloc}{{\rm Picloc }}
\newcommand{\Br}{{\rm Br}}
\newcommand{\Cl}{{\rm Cl}}
\newcommand{\NS}{{\rm NS}}
\newcommand{\id}{{\rm id}}
\newcommand{\Sym}{{\mathfrak S}}
\newcommand{\Alt}{{\mathfrak A}}
\newcommand{\Aut}{{\rm Aut}}
\newcommand{\Inn}{{\rm Inn}}
\newcommand{\Out}{{\rm Out}}
\newcommand{\Hol}{{\rm Hol}}
\newcommand{\Autp}{{\rm Aut}^p}
\newcommand{\End}{{\rm End}}
\newcommand{\Hom}{{\rm Hom}}
\newcommand{\Ext}{{\rm Ext}}
\newcommand{\ord}{{\rm ord}}
\newcommand{\perf}{{\rm pf}}
\newcommand{\coker}{{\rm coker}\,}
\newcommand{\divisor}{{\rm div}}
\newcommand{\nr}{{\rm ur}}
\newcommand{\Def}{{\rm Def}}
\newcommand{\et}{{\rm \acute{e}t}}
\newcommand{\loc}{{\rm loc}}
\newcommand{\ab}{{\rm ab}}
\newcommand{\sh}{{\rm sh}}
\newcommand{\piet}{{\pi_1^{\rm \acute{e}t}}}
\newcommand{\pietloc}{{\pi_{\rm loc}^{\rm \acute{e}t}}}
\newcommand{\piN}{{\pi^{\rm N}_1}}
\newcommand{\piNloc}{{\pi_{\rm loc}^{\rm N}}}
\newcommand{\Het}[1]{{H_{\rm \acute{e}t}^{{#1}}}}
\newcommand{\Hfl}[1]{{H_{\rm fl}^{{#1}}}}
\newcommand{\HZar}[1]{{H_{\rm Zar}^{{#1}}}}
\newcommand{\Hcris}[1]{{H_{\rm cris}^{{#1}}}}
\newcommand{\HdR}[1]{{H_{\rm dR}^{{#1}}}}
\newcommand{\hdR}[1]{{h_{\rm dR}^{{#1}}}}
\newcommand{\Torsloc}{{\rm Tors}_{\rm loc}}
\newcommand{\defin}[1]{{\bf #1}}
\newcommand{\oX}{\cal{X}}
\newcommand{\oA}{\cal{A}}
\newcommand{\oY}{\cal{Y}}
\newcommand{\calC}{{\cal{C}}}
\newcommand{\calL}{{\cal{L}}}
\newcommand{\bmu}{\boldsymbol{\mu}}
\newcommand{\balpha}{\boldsymbol{\alpha}}
\newcommand{\bL}{{\mathbf{L}}}
\newcommand{\bM}{{\mathbf{M}}}
\newcommand{\bW}{{\mathbf{W}}}
\newcommand{\bD}{{\mathbf{D}}}
\newcommand{\bT}{{\mathbf{T}}}
\newcommand{\bO}{{\mathbf{O}}}
\newcommand{\bI}{{\mathbf{I}}}
\newcommand{\AutScheme}{{\mathbf{Aut}}}
\newcommand{\InnScheme}{{\mathbf{Inn}}}
\newcommand{\OutScheme}{{\mathbf{Out}}}
\newcommand{\Isom}{{\mathbf{Isom}}}
\newcommand{\Norm}{{\mathbf{N}}}
\newcommand{\Cent}{{\mathbf{Z}}}
\newcommand{\Split}{{\mathbf{Split}}}
\newcommand{\BD}{{\mathbf{BD}}}
\newcommand{\BT}{{\mathbf{BT}}}
\newcommand{\BI}{{\mathbf{BI}}}
\newcommand{\BO}{{\mathbf{BO}}}
\newcommand{\C}{{\mathbf{C}}}
\newcommand{\Dic}{{\mathbf{Dic}}}
\newcommand{\SL}{{\mathbf{SL}}}
\newcommand{\PSL}{{\mathbf{PSL}}}
\newcommand{\PGL}{{\mathbf{PGL}}}
\newcommand{\MC}{{\mathbf{MC}}}
\newcommand{\GL}{{\mathbf{GL}}}
\newcommand{\Torus}{{\mathbf{T}}}
\newcommand{\Tors}{{\mathbf{Tors}}}

\newcommand{\bC}{{\mathbb{C}}}
\newcommand{\bR}{{\mathbb{R}}}

\newcommand{\Klein}{{\mathfrak{Klein}}}

\newif\ifhascomments \hascommentstrue
\ifhascomments
  \newcommand{\matt}[1]{{\color{red}[[\ensuremath{\spadesuit\spadesuit\spadesuit} #1]]}}
  \newcommand{\christian}[1]{{\color{red}[[\ensuremath{\star\star\star} #1]]}}
\else
  \newcommand{\matt}[1]{}
  \newcommand{\christian}[1]{}
\fi

\setcounter{tocdepth}{1}

\makeatletter
\@namedef{subjclassname@2020}{\textup{2020} Mathematics Subject Classification}
\makeatother

\title[An arithmetic analog of Klein's classification]{An arithmetic analog of Klein's classification of finite subgroups of $\SL_2(\CC)$}


\author{Christian Liedtke}
\address[CL]{TU M\"unchen, Zentrum Mathematik - M11, Boltzmannstr. 3, 85748 Garching bei M\"unchen, Germany}
\email{christian.liedtke@tum.de}

\author{Matthew Satriano}
\thanks{MS was partially supported by a Discovery Grant from the
  National Science and Engineering Research Council of Canada.}
\address[MS]{Department of Pure Mathematics, University
  of Waterloo, Waterloo ON N2L3G1, Canada}
\email{msatrian@uwaterloo.ca}


\begin{abstract}
Let $K$ be a number field with ring of integers $\OO_K$.
We describe and classify 
finite, flat, and linearly reductive subgroup schemes of $\SL_2$
over $\Spec\OO_K$. 
We also establish finiteness results for these
group schemes, as well as density results for the associated
quotient singularities.
\end{abstract}
\maketitle
\setcounter{tocdepth}{1}
\tableofcontents

\section{Introduction}
In his famous lectures on quintic equations and the icosahedron \cite{Klein},
Felix Klein classified finite subgroups of $\SL_2(\CC)$ and computed the
resulting quotient singularities.
It turns out that there are two infinite series ($A_n$
and $D_n$) and three
`exceptional' ($E_6$, $E_7$, $E_8$) of such groups and singularities.
The resulting singularities are precisely the 
\emph{rational double point singularities} 
(RDP singularities for short), 
which have shown up in many different contexts and which
can be characterized in many different ways, see, for example,
\cite{Durfee}.

The RDP singularities over arbitrary algebraically closed fields
have been classified by Michael Artin \cite{ArtinRDP},
who showed that the situation is more
complicated in characteristics $2,3,5$.
This was extended to nonclosed fields by 
Joseph Lipman \cite{Lipman},
where new types of singularities (non-split forms of the previous types)
like $B_n$, $C_n$, $F_4$, and $G_2$ arise.

If the ground field $k$ is of positive characteristic or if
it is not algebraically closed,
then it is usually not possible to obtain all types of 
RDP singularities as quotient singularities by finite groups 
of $\SL_{2,k}$.
However, it is possible to realize all types 
of RDP singularities (with some exceptions in characteristics $2,3,5$) 
as quotients by \emph{finite linearly
reductive subgroup schemes} of $\SL_{2,k}$,
which is due to Mitsuyasu Hashimoto \cite{Hashimoto} and 
the authors \cite{LS} over
algebraically closed ground fields and even holds 
over nonclosed fields, as recently shown by the authors 
in \cite{LS2}.

This fits into the general philosophy from \cite{LMM} 
that one should consequently use finite and linearly
reductive group schemes when working with quotient
singularities in positive or mixed characteristic.
For example, the second named author already established
a Chevalley--Shephard-Todd theorem  in this setting \cite{SatrianoCST}.
Another example is the existence of a very well-behaved 
McKay correspondence in positive characteristic 
in this setting that was
established by the first named author \cite{LieMcKay}.
Last, but not least, it also fits into the philosophy 
of \emph{tame stacks}, which were introduced by 
Abramovich, Olsson, and Vistoli \cite{AOV}, and where the
stabilizers are linearly reductive group schemes.

\subsection{The Klein set and arithmetic RDP singularities}
If $S$ is a scheme and $n\geq1$ is an integer, then we define
the \emph{Klein set} to be
$$
\Klein(n,S)
\quad:=\quad
\left\{ \begin{array}{l}
\mbox{finite, flat, linearly reductive subgroup  }\\
\mbox{schemes of } \SL_{2,S} \mbox{ of length }n
\end{array}
 \right\} / \sim,
$$
where $\sim$ is the equivalence relation of $\GL_{2}(S)$-conjugacy.
Associated to $G\subset\SL_{2,S}$ from $\Klein(n,S)$, we have
the invariant $\OO_S$-subalgebra
$$
\OO_S[x,y]^G \,\subseteq\,\OO_S[x,y],
$$
whose relative spectrum over $S$ is a family of RDP singularities
over $S$.

In the case where $S=\Spec k$ for some algebraically 
closed field $k$, the set $\Klein(n,S)$ and the associated RDPs
can be read off from the above mentioned works 
of Klein, Hashimoto, and ourselves \cite{Hashimoto, Klein, LS}.
If $S=\Spec k$ is the spectrum of an arbitrary field, then 
$\Klein(n,S)$ will be described in work in preparation
and the associated RDPs have been worked out if $k$ is perfect
already in \cite{LS2}.

The objective of this article is to first establish some
general properties of finite flat linearly reductive
group schemes over an excellent Dedekind scheme $S$
and to use this to give a cohomological description
of $\Klein(n,S)$.

The main application is to establish finiteness results
and to give a more or less explicit descriptions of
$\Klein(n,S)$ in the case where $S$ is the spectrum 
of the ring of integers $\OO_K$ inside a number field $K$.

This can be viewed as an arithmetic analogue of 
Klein's classification of finite subgroups of $\SL_2(\CC)$
and their associated quotient singularities.

\subsection{Subgroup schemes of $\SL_{2,\OO_K}$}
The key observation is the following result, which follows from
the more general Theorem \ref{thm: key}.

\begin{Theorem}
    Let $K$ be a number field with ring of integers $\OO_K$.
    Let $G\subset\SL_{2,\OO_K}$
    be a finite flat linearly reductive subgroup scheme
    of length $n$ over $\Spec\OO_K$.
    Then, there exists an fppf cover $\{S'_i\}_{i\in I}$
    of $\Spec\OO_K$, such that $G_{S'_i}$
    is conjugate to the standard embedding 
    $\bmu_{n,S_i'}\subset\GL_{2,S'_i}$.
\end{Theorem}

Thus, $G\subset\SL_{2,\OO_K}$ is locally in the fppf topology
conjugate  to the standard embedding $\bmu_n\subset\SL_{2,\OO_K}$.
The reason why only forms of $\bmu_n\subset\SL_2$ show up
and none other of the group schemes from Klein's list over $\CC$,
is the following:
let $\idealp\subset\OO_K$ be a prime lying over $(2)\subset\ZZ$, 
let $\kappa(\idealp)$
be the residue field, and let $\overline{k}$ be an algebraic closure of
$\kappa(\idealp)$.
Then, $(G_\idealp)_{\overline{k}}\subset\SL_{2,{\overline{k}}}$ must be conjugate
to the standard embedding
$\bmu_n\subset\SL_{2,k}$ by the classification of finite linearly
reductive subgroup schemes of $\SL_2$ over algebraically closed fields
of characteristic $2$, which we recall in Theorem \ref{thm: Klein}.
From this, Theorem \ref{thm: G over Dedekind}, which is a rigidity
result for finite flat and linearly reductive subgroup schemes, 
shows that $G$ itself must be a twisted form of the standard embedding
$\bmu_{n,\OO_K}\subset\SL_{2,\OO_K}$.

\subsection{Classifying twisted forms of $\bmu_n$}
The classification of finite flat and linearly reductive
subgroup schemes of $\SL_{2,\OO_K}$ is thus equivalent
to the classification of fppf locally conjugate forms
of the standard embedding $\bmu_n\to\SL_{2,\OO_K}$.
To obtain the latter, we note that the
normalizers of $\bmu_n$ inside $\SL_2$ and $\GL_2$
(over any base scheme) sit in split exact sequences
\begin{equation}
    \label{eq: normalizer intro}
    \begin{array}{ccccccccc}
     1&\to&\GG_m&\to&\Norm_{\SL_2}(\bmu_n)&\to&\ZZ/2\ZZ&\to&1 \\
     1&\to&\GG_m^2&\to&\Norm_{\GL_2}(\bmu_n)&\to&\ZZ/2\ZZ&\to&1 
     \end{array}
\end{equation}
of affine group schemes over $\OO_K$,
where $\GG_m\subset\SL_2$ and $\GG_m^2\subset\GL_2$ are the standard tori.
Extending \cite[Theorem 1.4]{LS2} to more general base schemes,
we have the following (see also Theorem \ref{thm: brian}).

\begin{Theorem}
 Twisted forms of standard embedding $\bmu_n\subset\SL_2$ in the fppf topology 
 are in bijection (up to $\GL_{2,\OO_K}$-conjugacy) to the kernel
 $$
   \ker\left( 
   \Hfl{1}(\OO_K,\Norm_{\GL_2}(\bmu_n)) \,\to\,
   \Hfl{1}(\OO_K,\GL_{2,\OO_K})\right),
 $$
 of pointed cohomology sets.
\end{Theorem}

Since the sequences \eqref{eq: normalizer intro}
split, we obtain a short exact sequence
of pointed cohomology sets
\begin{equation}
\label{eq: cohomology intro}
1\,\to\,
\Hfl{1}(\OO_K,\GG_m^2)\,\to\, \Hfl{1}(\OO_K,\Norm_{\GL_2}(\bmu_n))
\,\to\,\Hfl{1}(\OO_K,\ZZ/2\ZZ)
\,\to\,1.
\end{equation}
In Section \ref{sec: twisted forms of mun}, we will
analyze the twisted forms of associated to cohomology classes
in this sequence over arbitrary base schemes:
it turns out that classes on the right lead to
certain quadratic twists and that classes on the left
lead to forms that can be trivialized Zariski locally.

\subsection{Finiteness} 
We note that
$$
\Hfl{1}(\OO_K,\GG_m)\quad\cong\quad\Pic(\OO_K)\quad\cong\quad\Cl(\OO_K)\,=\,:\Cl_K,
$$
the class group of $K$.
Moreover, $\Hfl{1}(\OO_K,\ZZ/2\ZZ)$ describes
unramified extensions of $\OO_K$ of degree 2,
see also Caution \ref{caution cft}.
Class field theory gives an isomorphism
$$
\Hfl{1}(\OO_K,\ZZ/2\ZZ)\quad\cong\quad \Hom(\Cl_K^1,\ZZ/2\ZZ),
$$
where $\Cl^1_K$ denotes
the ray class group of modulus 1, which can be described by
$\Cl_K$ and the real places of $K$.
In particular, the kernel and the cokernel of
sequence \eqref{eq: cohomology intro}
are controlled by $\Cl_K$ and $\Cl_K^1$, respectively.

Since $\Cl_K$ and $\Cl^1_K$ both are finite (recall
that the class number $h_K$ is defined to be the order of $\Cl_K$),
we eventually obtain the following finiteness result and we note
that it does not immediate follow from the previous observations
due to subtleties of non-abelian cohomology.

\begin{Theorem}
Let $K$ be a number field with ring of integers $\OO_K$, 
class number $h_K$,
and ray class group $\Cl_K^1$ of modulus 1.
Then, $\Klein(n,\OO_K)$ is a finite set.
More precisely,
\begin{enumerate}
 \item If $n=2$, then $\Klein(2,\OO_K)$ is a singleton.
 \item If $n\geq3$, then
 \begin{enumerate}
    \item  the cardinality $|\Klein(n,\OO_K)|$ is independent of $n$.
    \item  $|\Klein(n,\OO_K)|\geq h_K$.
    \item  $\Klein(n,\OO_K)$ consists of one element if and only if
      $\Cl_K^1$  is trivial.
      \end{enumerate}
\end{enumerate}
\end{Theorem}

We mention that this result is slightly in the spirit of
other finiteness theorems:
for example, Frans Oort and John Tate showed that the only finite group scheme 
of length $p$ over $\Spec \ZZ$ is $\ZZ/p\ZZ$ or $\bmu_p$, 
see the introduction of \cite{OortTate}, where they actually 
attribute this to unpublished work of Michael Artin and Barry Mazur.
Moreover,  Jean-Marc Fontaine \cite{Fontaine} studied
finite subgroup schemes of a potential abelian scheme over $\Spec\ZZ$
and established finiteness in order to show that there is no abelian scheme 
over $\Spec\ZZ$.

\subsection{RDP singularities over $\OO_K$}
For an integer $n\geq2$, let $\beta(n)$ be equal to $n/2$ if $n$ is even 
and equal to $(n-1)/2$ if $n$ is odd.

\begin{Theorem}
    Let $K$ be a number field with ring of integers 
    $\OO_K$, let $G\in\Klein(n,\OO_K)$ for some $n\geq3$,
    and let
    $$ 
    f\quad:\quad
\Spec \OO_K[x,y]^G \quad\to\quad\Spec\OO_K
    $$
    be the associated family of RDP singularities 
    over $\Spec\OO_K$.
 \begin{enumerate}
 \item If $G\cong\bmu_{n,\OO_K}$ (abstract isomorphism of group schemes, disregarding the 
 embedding into $\SL_{2,\OO_K}$), 
 then for all $\idealp\in\Spec\OO_K$,
 we have $G_{\kappa(\idealp)}\cong\bmu_{n,\kappa(\idealp)}$ and
 $f^{-1}(\idealp)$
 is an RDP of type $A_{n-1}$ over $\kappa(\idealp)$.
 \item If $G\not\cong\bmu_{n,\OO_K}$ (abstract isomorphism of group schemes, 
 disregarding the embedding into $\SL_{2,\OO_K}$), then there exists an 
 unramified quadratic extension $L/K$ (see Caution \ref{caution cft}), such
 such that 
 $$
  G \quad\cong\quad\bmu_{n,\OO_K}\wedge\Spec\OO_L,
 $$
 the quadratic twist with respect to inversion.
 Let $\Sigma_A$ (resp. $\Sigma_B$) be the subset of $\Spec\OO_K$ 
 of primes that are inert (resp. split) in $\OO_L$.
 \begin{enumerate}
 \item If $\idealp\in\Sigma_A$, then $G_{\kappa(\idealp)}\cong\bmu_{n,\kappa(\idealp)}$
 and $f^{-1}(\idealp)$ is an RDP of type $A_{n-1}$
 over $\kappa(\idealp)$.
  \item If $\idealp\in\Sigma_B$, then $G_{\kappa(\idealp)}\not\cong\bmu_{n,\kappa(\idealp)}$
 and $f^{-1}(\idealp)$ is and RDP 
 of type $B_{\beta(n)}$ over $\kappa(\idealp)$.
 \end{enumerate}
 Both, $\Sigma_A$ and $\Sigma_B$ have Dirichlet 
 density $1/2$ in $\Spec\OO_K$
 and in particular, they are infinite.
 \end{enumerate}
\end{Theorem}

We refer to Section \ref{sec: number fields}
for more details and for what happens at 
the infinite places.

\begin{Remark}
In Section \ref{sec: examples}, we give two examples
of non-standard embeddings of $\bmu_{n,\OO_K}$ into 
$\SL_{2,\OO_K}$.
In Example \ref{ex:more-example-non-stnd-mun-invars}, we compute
the invariant subring in one of these examples, which exhibits interesting behavior. 
It is a family of RDPs of type $A_{n-1}$ over $\Spec\OO_K$.
It turns out that in this example
Noether's degree bound on invariants of 
$\OO_K[x,y]^{\bmu_n}\subset\OO_K[x,y]$ holds, 
see Remark \ref{rem: noether}.
However, it seems that more than 3 invariants
are needed to generate the invariant subring as a
$\OO_K$-algebra.
(We believe that $5$ invariants are needed in this example.)
Interestingly, one can find a Zariski open cover, such that
the invariant subring is generated by
3 invariants on each patch.
\end{Remark}

\subsection{Organization}
This article is organized as follows:
\begin{itemize}
\item In Section \ref{sec: generalities}, we recall some generalities
on finite linearly reductive group schemes.
\item In Section \ref{sec: Klein classification}, we recall 
the classification of finite linearly reductive subgroup
schemes of $\GL_{d,k}$ and of $\SL_{2,k}$ in the case where
$k$ is an algebraically closed field.
We also introduce the Klein set $\Klein(n,S)$.
\item In Section \ref{sec: rigidity}, we establish a rigidity
result that is central to what follows:
if $S$ is an excellent Dedekind scheme, if $s\in S$ is a point, 
if $G_i\subset\GL_{d,S}$, $i=1,2$ are two finite flat linearly
reductive subgroup schemes, and if $G_{1,s}$ and $G_{2,s}$ are conjugate
in $\GL_{d,s}$, then $G_1$ and $G_2$ are locally in the fppf topology
conjugate to each other inside $\GL_{d,S}$.
\item In Section \ref{sec: twisted forms}, we study subgroup
schemes $H\subset X$ that are fppf locally conjugate
to some closed subgroup scheme $G\subset X$, where $G$ and $X$
are affine group schemes over some base
scheme $S$.
\item In Section \ref{sec: key}, we combine the rigidity result
of Section \ref{sec: rigidity} with the fact that the only
linearly reductive subgroup schemes of $\SL_{2,k}$, where
$k$ is an algebraically closed field of characteristic 2,
are those that are conjugate to $\bmu_{n,k}\subset\SL_{2,k}$.
We conclude that if $S$ is an excellent Dedekind scheme
that has a point of residue characteristic 2, then
the only finite flat linearly reductive subgroup
schemes of $\SL_{2,S}$ are twisted forms
of the standard embedding $\bmu_{n,S}\subset\SL_{2,S}$.
\item In Section \ref{sec: twisted forms of mun},
we analyze subgroup schemes of $\SL_{2,S}$ that are locally
in the fppf topology conjugate to the standard embedding
$\bmu_{n,S}\subset\SL_{2,S}$ over some
base scheme $S$.
We also explicitly describe forms that are locally trivial 
for the Zariski topology (Section \ref{subsec: Zariski})
and forms that are trivial 
for the \'etale topology (Section \ref{subsec: etale}).
\item In Section \ref{sec: examples}, we give an explicit
example of a closed subgroup scheme $G\subset\SL_{2,\OO_K}$ 
that is Zariski locally conjugate to the standard embedding 
$\bmu_{n,\OO_K}\subset\SL_{2,\OO_K}$,
but not globally $\SL_2(\OO_K)$-conjugate.
Here, $\OO_K$ is the ring of integers in a certain number
field $K$.
We also compute the invariant subring in an example.
\item In Section \ref{sec: number fields}, we apply our
results to rings of integers $\OO_K$ in number fields.
We establish finiteness of $\Klein(n,\OO_K)$
and density results for the associated families
of RDP singularities over $\OO_K$.
\item In Section \ref{sec: outlook}, we give an outlook
on $\Klein(n,\OO_{K,\Sigma})$ if one 
allows a finite set $\Sigma$ of primes of bad reduction.
\end{itemize}

\begin{Acknowledgements}
We thank Dori Bejleri, Alessandro Cobbe, Werner Bley, Gregor Kemper, David McKinnon, 
and Jerry Wang for discussions and help.
\end{Acknowledgements}

\begin{Conventionsx}
 In this article, all schemes, group schemes, and torsors under group schemes will be of 
 finite type over Noetherian separated base schemes.
 In particular, all torsors in the fpqc topology can already 
 be trivialized in the fppf topology.
 We will work in the fppf topology and flat cohomology will refer
 to the fppf topology.
\end{Conventionsx}

\section{Generalities}
\label{sec: generalities}

In this section, we recall several more or less well-known 
results on linearly reductive
group schemes, first over arbitrary base schemes and then, over fields
with an emphasis on algebraically closed fields.
We follow the setup of Abramovich--Olsson--Vistoli from \cite{AOV}
and many results over arbitrary base schemes are in fact due to them.

\subsection{Over arbitrary bases}
Let $G\to S$ be a group scheme over a base scheme $S$.
We denote by $\mathrm{QCoh}(-)$ (resp. $\mathrm{Coh}(-)$) the category
of quasi-coherent (resp. coherent sheaves) and a supscript $-^G$
refers to the category of $G$-equivariant such objects and morphisms.
Following \cite[Definition 2.2]{AOV},
we say that $G$ is \emph{linearly reductive} if the functor
of taking $G$-invariants
$$
\mathrm{QCoh}^G(S) \,\to\, \mathrm{QCoh}(S),
\qquad \mathcal{F}\,\mapsto\,\mathcal{F}^G
$$
is exact.
If $S$ is Noetherian, which is our standing assumption, 
then this is equivalent to the functor
$$
\mathrm{Coh}^G(S) \,\to\, \mathrm{Coh}(S),
\qquad \mathcal{F}\,\mapsto\,\mathcal{F}^G
$$
being exact, see \cite[Proposition 2.3]{AOV}.
The following result shows that linearly reductive group
schemes behave well under base-change.

\begin{Proposition}[Abramovich--Olsson--Vistoli]
\label{prop: base change}
 Let $S'\to S$ be a morphism of schemes, let
 $G\to S$ be a group scheme, and let $G_S':=G\times_SS'\to S'$ be the 
 base change.
 Then
 \begin{enumerate}
 \item If $G\to S$ is linearly reductive, then $G'\to S'$ is linearly reductive.
 \item If $G'\to S'$ is linearly reductive and if $S'\to S$ is flat and surjective, then
 $G\to S$ is linearly reductive.
 \end{enumerate}
\end{Proposition}

\begin{proof}
\cite[Proposition 2.4]{AOV}.
\end{proof}

In particular, if $k\subseteq K$ is an extension of fields
and if $G\to\Spec k$ is a group scheme, then $G\to\Spec k$ is linearly
reductive if and only if $G\times_{\Spec k}\Spec K\to\Spec K$ is 
linearly reductive.

\begin{Proposition}[Abramovich--Olsson--Vistoli]
The class of finite linearly reductive group schemes is closed under taking
subgroup schemes, quotients, and extensions.
\end{Proposition}

\begin{proof}
\cite[Proposition 2.5]{AOV}.
\end{proof}

\begin{Lemma}
\label{lemma: aov}
    Let $k$ be a field and let $G\to\Spec k$ be a finite group scheme.
    Then, the following are equivalent:
    \begin{enumerate}
        \item $G\to\Spec k$ is linearly reductive in the sense 
        of \cite{AOV}.
        \item Every short exact sequence 
        \begin{equation}
\label{eq: aov}
 0\,\to\,V'\,\to\,V\,\to\,V''\,\to\,0
\end{equation}
  of finite-dimensional $k$-linear $G$-representations 
  splits as a sequence of $k$-linear $G$-representations.
  \item Every finite dimensional $k$-linear $G$-representation
  is a direct sum of simple $k$-linear $G$-representations.
    \end{enumerate}
\end{Lemma}

\begin{proof}
$(1)\Rightarrow(2)$
We consider the sequence \eqref{eq: aov}.
Clearly $\Hom(V'',-)$ is an exact functor because
we are dealing with finite dimensional $k$-vector spaces.
Taking $G$-invariants is an exact functor because $G$ is assumed
to be linearly reductive.
This shows that $\Hom(V'',-)^G=\Hom_G(V'',-)$ is an exact functor.
Applying this functor to \eqref{eq: aov}, we conclude that 
there exists a $\sigma\in\Hom_G(V'',V)$ mapping to $\id_{V''}\in\Hom_G(V'',V'')$.
This $\sigma$ is a splitting of \eqref{eq: aov} in the category
of $k$-linear $G$-representations.

$(2)\Leftrightarrow(3)$ is an exercise.

$(3)\Rightarrow(1)$
Given a short exact sequence \eqref{eq: aov},
we can decompose every vector space into the direct sum of
the isotypical part of the trivial representation and the other ones.
Taking $G$-invariants and using Schur's lemma, 
we obtain a short exact sequence
$$
 0\,\to\,V'^G\,\to\,V^G\,\to\,V''^G\,\to\,0
$$
of finite dimensional $k$-vector spaces.
This shows that taking $G$-invariants 
$\mathrm{Coh}^G(\Spec k)\to\mathrm{Coh}(\Spec k)$
is an exact functor.
By \cite[Proposition 2.3]{AOV}, this implies that taking
$G$-invariants  $\mathrm{QCoh}^G(\Spec k)\to\mathrm{QCoh}(\Spec k)$
is an exact functor, that is, $G\to\Spec k$ is linearly
reductive.
\end{proof}

\begin{Remark}
Classically, a finite group scheme $G\to\Spec k$ is said
to be linearly reductive if (3) of this lemma 
holds.
\end{Remark}

\begin{Definition}
 Let $S$ be a scheme.
 \begin{enumerate}
 \item A finite and \'etale group scheme $G\to S$ is called \emph{tame} if 
 its degree is prime to all residue characteristics.
 \item A finite group scheme $G\to S$ is called \emph{constant} if it is 
 isomorphic to the group scheme
 $$
\coprod_{g\in G_{\mathrm{abs}}} S \,\to\,S
 $$
 associated to some finite group $G_{\mathrm{abs}}$. 
 We will also say that $G$ is the \emph{constant group scheme associated to 
 the group} $G_{\mathrm{abs}}$ and write the latter
 as $(G_{\mathrm{abs}})_S\to S$.

 A finite group scheme $G\to S$ is called \emph{locally constant}
 if there exists an fppf cover $\{S_j'\}_{j\in J}$ of $S$
 such that $G_{S'_j}\to S_j'$ is constant for all $j\in J$.
 \item A finite and commutative group scheme $G\to S$ is called
 \emph{(locally) diagonalizable} if its Cartier dual $\shHom(G,\GG_{m,S})\to S$ is
 (locally) constant.
 \end{enumerate}
\end{Definition}

Clearly, if $G\to S$ is a finite and (locally) constant group scheme,
then it is \'etale.
We refer to \cite[Section 2.3]{AOV} for more details and results 
about these classes of group schemes.
The most important example for this article is the following.

\begin{Example}
 Let $S$ be a scheme.
 \begin{enumerate}
 \item If $G\to S$ is locally constant and tame, then it is
 linearly reductive.
 \item The group scheme of $n$.th roots of unity
 $\bmu_{n,S}\to S$,
 which is Cartier dual to the constant group scheme 
 $(\ZZ/n\ZZ)_S\to S$,
 is diagonalizable and linearly reductive.
 \end{enumerate}
 Both assertions follow from \cite[Proposition 2.7]{AOV}.
\end{Example}

\subsection{Over algebraically closed fields}
If $G$ is the constant group scheme over $k$ 
associated to a finite group $G_{\mathrm{abs}}$, then
$G$ is linearly reductive if and only if $G_{\mathrm{abs}}$ is of order prime to $p$.
In particular, if $p=0$, then $G$ is automatically linearly reductive.
Thus, one may think of non-modular representation theory of finite groups
over $k$ as a special case of the representation theory of finite and linearly 
reductive group schemes.

\begin{Theorem}
 Let $G$ be a finite group scheme
 over an algebraically closed field $k$
 of characteristic $p\geq0$.
 Then, the following are equivalent:
 \begin{enumerate}
  \item $G$ is linearly reductive.
  \item $G$ is isomorphic to an extension
  of finite group schemes over $k$
  $$1\,\to\,D\,\to\,G\,\to\,F\,\to\,1,$$
  where $F$ is tame and where
  $D$ is diagonalizable.
  \item There exists an isomorphism of group schemes over $k$
  $$
 G\quad\cong\quad \prod_{i=1}^m\bmu_{n,p^{m_i}}\,\rtimes\, H,
  $$
  where $H$ is the constant group scheme associated to some 
  finite group of order prime to $p$.
 \end{enumerate}
\end{Theorem}

\begin{Remark}
This result is often attributed to Nagata \cite{Nagata}.
However, Nagata assumed $G$ to be affine and connected.
It also applies to the non-connected case, which has been
established independently of each other 
by Abramovich--Olsson--Vistoli \cite{AOV}, 
by Chin \cite{Chin}, and by Hashimoto \cite{Hashimoto}.
\end{Remark}

This classification fits as follows into the general
structure theory for finite group schemes over fields:
if $G\to \Spec k$ is a finite group scheme over a field $k$,
then there exists a short exact sequence of group schemes
over $k$, the
\emph{connected-\'etale sequence}
\begin{equation}
\label{eq: connected etale}
1\,\to\,G^\circ\,\to\,G\,\to\,G^\et\,\to\,1,
\end{equation}
where $G^\circ$ is the connected component of the identity
and $G^\et$ is the maximal \'etale quotient.
If $k$ is perfect, then this sequence splits, and we obtain 
a semi-direct product decomposition $G\cong G^\circ\rtimes G^\et$.
If $k$ is algebraically closed, then 
$G$ is linearly reductive if and only if 
$G^\circ$ is diagonalizable and $G^\et$ is tame.
We refer to \cite[Section 2.3]{AOV} for results about
local semi-direct product decompositions for linearly 
reductive group schemes over arbitrary base schemes.

\section{Klein's classification over fields}
\label{sec: Klein classification}

In this section, we recall the classification of finite and
linearly reductive subgroup schemes of $\SL_{2,k}$, where
$k$ is an algebraically closed field.
We also recall a general lifting result, which explains
why the classification list in characteristic $p\gg0$
is `the same' as over the complex numbers.

\subsection{Generalities}
First, we recall \cite[Definition 2.7]{LMM}:
If $G\to\Spec k$ is a finite and linearly reductive group scheme over an 
algebraically closed field $k$ and if $\rho:G\to\GL(V)$  
is a finite dimensional $k$-linear representation, then we define 
$$
\lambda(\rho) \,:=\,\max_{\{\mathrm{id}\}\neq\bmu_n\subseteq G} \dim_kV^{\bmu_n},
$$
where $V^{\bmu_n}\subseteq V$ denotes the subspace of $\bmu_n$-invariants.
As observed in \cite[Remark 2.8]{LMM},  we have that
$$\begin{array}{lcl}
\rho \mbox{ is faithful } &\Leftrightarrow& \lambda(\rho)\leq\dim V-1 \\
\rho \mbox{ contains no pseudo-reflections }&\Leftrightarrow& \lambda(\rho)\leq\dim V-2 .
\end{array}$$
In the second case, one also says that $\rho(G)$ is a \emph{small} subgroup scheme of $\GL(V)$.
It is called \emph{very small} if $\lambda(\rho)=0$.
The importance of smallness stems from the classical theorems 
of Chevalley and Shephard--Todd, which have been extended to
linearly reductive group schemes in \cite{SatrianoCST}.

Regarding the classification of finite and linearly reductive subgroup schemes of
$\GL_{d,k}$, we have the following reduction to the complex numbers.

\begin{Theorem}[Liedtke--Martin--Matsumoto]
\label{thm: LMM lifting}
  Let $k$ be an algebraically closed field of characteristic $p>0$.
  Then, lifting induces a canonical map
  $$
  \begin{array}{cl}
  &\{\mbox{ finite and linearly reductive subgroup schemes of }\GL_{d,k}\mbox{ up to conjugacy }\} \\
  \rightarrow&\{\mbox{ finite subgroups of }\GL_{d}(\CC)\mbox{ up to conjugacy }\},
  \end{array}
  $$
  which has the following properties:
  \begin{enumerate}
  \item It is injective.
  \item It is surjective if $p>2d+1$.
  \item The image is the set of conjugacy classes of finite subgroups of $\GL_{d}(\CC)$
  that have a unique (possibly trivial) abelian $p$-Sylow subgroup.
  \item It maps subgroup schemes of $\SL_{d,k}$ to subgroups of $\SL_{d}(\CC)$.
  \item It preserves $\lambda$-invariants. In particular, small (resp. very small) subgroup schemes
  get mapped to small (resp. very small) subgroups.
  \end{enumerate}
\end{Theorem}

\begin{proof}
    \cite[Theorem 2.13]{LMM}.
\end{proof}

\subsection{The classification list of Hashimoto and Klein}
\label{subsec: Hashimoto and Klein}

The finite subgroups of $\SL_2(\CC)$ have been classified by Klein \cite{Klein}
and they are automatically small.
This classification was extended to algebraically closed fields
of positive characteristic by Hashimoto \cite{Hashimoto}.
Alternatively, one can use
Theorem \ref{thm: LMM lifting} to obtain a new proof of the latter classification 
- it also shows that the classification of finite linearly
reductive subgroup schemes of $\SL_{2,k}$ or $\GL_{2,k}$, 
where $k$ is an algebraically
closed field of characteristic $p>5$ will be `the same' as over the complex numbers.

\begin{Theorem}[Hashimoto, Klein]
\label{thm: Klein}
Let $G$ be a finite and linearly reductive subgroup scheme of $\SL_{2,k}$,
where $k$ is an algebraically closed field of characteristic $p\geq0$.
Then, $G$ is conjugate to one of the following.
\begin{enumerate}
\item $(n\geq1$) The group scheme $\bmu_n$ of length $n$, which is
generated as a subgroup scheme of $\SL_{2,k}$ by
$$
\left(\begin{array}{cc}
\zeta_n & 0 \\
0&\zeta_n^{-1}
\end{array}\right)
$$
where $\zeta_n$ denotes a primitive $n$.root of unity.
\item $(n\geq2, p\neq2)$ The binary dihedral group scheme $\BD_n$ of length $4n$
generated by $\bmu_{2n}\subset\SL_{2,k}$ as in (1) and 
$$
\left(\begin{array}{cc}
0 & \zeta_4 \\
\zeta_4 & 0
\end{array}\right)
$$
inside $\SL_{2,k}$.
\item $(p\neq2,3)$ The binary tetrahedral group scheme $\BT$ of length $24$.
\item $(p\neq2,3)$ The binary octahedral group scheme $\BO$ of length $48$.
\item $(p\neq2,3,5)$ The binary icosahedral group scheme $\BI$ of length $120$.
\end{enumerate}
\end{Theorem}

\begin{Remark}
In particular, we see that  if $p=2$, then only case (1) exists.
This will be crucial in the sequel.
\end{Remark}

\subsection{Beyond Klein}
\label{subsec: beyond Klein}
The finite and \emph{small} subgroups of $\GL_2(\CC)$ were classified by
Brieskorn \cite[Satz 2.9]{Brieskorn}.
Using Theorem \ref{thm: LMM lifting}, we obtain the classification
of finite, linearly reductive, and small subgroup schemes of $\GL_{2,k}$
where $k$ is an algebraically closed field of characteristic $p>0$.
We refer to \cite[Theorem 3.4]{LMM} for this list.

\begin{Example}
\label{ex: bmunq}
 Let $k$ be an algebraically closed field of characteristic $p\geq0$.
 Let $n\geq2$ and $1\leq q<n$ be integers and assume that $q$ and $n$ are coprime.
 Then, we denote by $\bmu_{n,q}$ the group scheme $\bmu_n$ as the subgroup scheme
 together with its embedding into $\GL_{2,k}$ generated by 
 $$
 \left(\begin{array}{cc}
 \zeta_n&0\\
 0&\zeta_n^{q}
 \end{array}
 \right),
 $$
 to be taken with a grain of salt if $p\mid n$.
 One can show that $\bmu_{n,q}$ is conjugate to $\bmu_{n',q'}$ if and only
 if $n=n'$ and $qq'\equiv1\mod n$.
 Moreover, if $q=n-1$, then we obtain the standard embedding of $\bmu_n$ into 
 $\SL_{2,k}$.
\end{Example}

\begin{Remark}
 If $p=2$, then it is a consequence of \cite[Theorem 3.4]{LMM} that 
 the only finite, linearly reductive, and small
 subgroup schemes of $\GL_{2,k}$ are the $\bmu_{n,q}$ from
 Example \ref{ex: bmunq}.
\end{Remark}

Over the complex numbers, the classification of \emph{very small} subgroups
of $\GL_{d}(\CC)$ is closely related to the so-called spherical form problem
in differential topology and we refer to \cite{Wolf} for an overview.
Using Theorem \ref{thm: LMM lifting}, these results can be used to 
classify finite, linearly reductive, and very small subgroup schemes of
$\GL_{d,k}$, where $k$ is an algebraically closed field of characteristic $p>0$.
We refer to \cite[Theorem 3.4]{LMM} for this classification.
Again, it turns out that there are particularly few cases if $p=2$: 
namely, a special class of metacyclic group schemes.

\subsection{The Klein set}
\label{subsec: Klein set}
For a scheme $S$, we define the \emph{Klein set}
$$
\Klein(n,S) \,:=\, 
\left\{ \begin{array}{l}
\mbox{finite, flat, linearly reductive subgroup  }\\
\mbox{schemes of } \SL_{2,S} \mbox{ of length }n
\end{array}
 \right\} / \sim,
$$
where $\sim$ is the equivalence relation of
$\GL_{2}(S)$-conjugacy.
If $R$ is a commutative ring, then we set
$\Klein(n,R):=\Klein(n,\Spec R)$.

If $k$ is an algebraically closed field, 
then Theorem \ref{thm: Klein} describes
$\Klein(n,k)$.
If $k$ is an arbitrary field, then we will describe
$\Klein(n,k)$ in a forthcoming monograph.
The associated quotient singularities have already 
been determined in \cite{LS2}.

\begin{Remark}
If $k$ is an algebraically closed field, then
$\GL_{2}(k)$-conjugacy and $\SL_{2}(k)$-conjugacy
of finite subgroup schemes of $\SL_{2,k}$ coincide
because if two subgroup schemes of $\SL_{2,k}$
are conjugate by a matrix
$A\in\GL_{2}(k)$, then there exists an $a\in k$
with $a^2=\det(A)$ and then, the two subgroup schemes
are conjugate by
$a^{-1}\cdot A$, which lies in $\SL_2(k)$.

However, if $S$ is a general base scheme, then there may
not exist a square root of $\det(A)$ in $\OO_S(S)$ and thus,
there may be
more $\SL_2(S)$-conjugacy classes of finite subgroup schemes
of $\SL_{2,S}$ than there are $\GL_2(S)$-conjugacy classes.
\end{Remark}

The next result is well-known in the case of finite subgroups 
of $\GL_m(\CC)$
and the following extension is probably well-known as well.

\begin{Lemma}
\label{lem: invariant well-defined}
  Let $R$ be a ring and let $G_i\to\Spec R$, $i=1,2$ be two finite and closed
  subgroup schemes of $\GL_{m,R}$ over $\Spec R$.
  If they are $\GL_m(R)$-conjugate, then the invariant subrings
  $R[x_1,...,x_m]^{G_i}$ with $i=1,2$ (resp. $R[[x_1,...,x_m]]^{G_i})$ 
  are isomorphic as $R$-algebras.
\end{Lemma}

\begin{proof}
Let $A\in\GL_m(R)$ with $AG_1A^{-1}=G_2$.
Let $V:=k^m$ be the $k$-vector space with basis $e_1,...,e_m$ and with its
natural $\GL_{m,R}$-action.
Multiplication by $A$ on the left defines an $R$-algebra automorphism of $R[V]$.
If $f\in R[V]$ is $G_1$-invariant,
then $Af\in R[V]$ is $G_2$-invariant.
Thus, left multiplication by $A$ induces an $R$-algebra homomorphism
$R[V]^{G_1}\to R[V]^{G_2}$, whose inverse is given by left multiplication
by $A^{-1}$.

Similarly, we obtain an $R$-algebra isomorphism 
$R[[V]]^{G_1}\to R[[V]]^{G_2}$.
\end{proof}

\begin{Definition}
For a ring $R$, an integer $n\geq1$, and a class $X\in\Klein(n,R)$,
which is represented by some $G\subset\SL_{2,R}$, 
we call $\Spec R[x,y]^G$ the \emph{(linearly reductive) quotient singularity
associated to $X$}.
\end{Definition}

By Lemma \ref{lem: invariant well-defined}, this is well-defined:
another choice of representative of the class $X\in\Klein(n,R)$ leads
to an isomorphic singularity.

\section{Rigidity over Dedekind schemes}
\label{sec: rigidity}

In this section, we establish some rigidity results
for finite flat linearly reductive group schemes
and their representations.
Using deformation theory, we first
treat the case where the base is the spectrum 
of complete DVR and then, 
where it is an excellent Dedekind scheme.
We expect some of our results to hold over more general
base schemes, but for the purposes of this article,
it suffices to treat excellent Dedekind schemes.

\subsection{Complete DVRs}
We start with results over complete DVRs, 
which are  close to the rigidity results in \cite[Section 2]{LMM}.
These are somewhat clumsy to state
since we wish to have a version that works both,
in equal and in mixed characteristic.
We start by recalling a couple of results and a 
little bit of terminology from \cite{LMM}, which we will need in the sequel,
but mostly in the proofs.

\begin{Proposition}[Liedtke--Martin--Matsumoto]
let $G\to \Spec k$ be a finite linearly reductive group
scheme over an algebraically closed field $k$ of characteristic $p\geq0$.
\begin{enumerate}
    \item Taking $k$-rational points, we obtain a finite group
    $$
       G_{\mathrm{abs}} \,:=\, (G^\circ)^D(k)\rtimes G^\et(k),
    $$
    where $G^\circ$ and $G^\et$ are as in \eqref{eq: connected etale},
    and where $-^D:=\shHom(-,\GG_m)$ denotes the Cartier dual of
    commutative group schemes over $k$.
    The assignment
    $$
     G\quad\mapsto\quad G_{\mathrm{abs}}
    $$
    induces an equivalence of categories between finite and linearly
    reductive group schemes over $k$
    and finite groups with a unique (possibly trivial) abelian $p$-Sylow subgroup.
    \item If $p=0$, then $G$ is isomorphic to the constant group scheme associated 
    to $G_{\mathrm{abs}}$.
    \item If $p>0$, then 
    \begin{enumerate}
    \item there exists a distinguished finite flat group scheme 
    $\mathcal{G}_{\mathrm{can}}\to\Spec W(k)$
    over the ring $W(k)$ of Witt vectors of $k$ with special fiber $G\to\Spec k$.
    \item If $R$ is a complete DVR with residue field $k$, whose
    field of fractions $K$ is of characteristic zero,
    if $K\subset\overline{K}$ is an algebraic closure, and if
    $\mathcal{G}\to\Spec R$ is a finite flat group scheme over $R$ 
    with special fiber $G$,
    then $G_{\mathrm{abs}}\cong\mathcal{G}(\overline{K})$, the group of
    $\overline{K}$-rational points.    
    \item If $R$ is a complete DVR with residue field $k$ and if
    $\mathcal{G}_i\to\Spec R$, $i=1,2$ are two finite flat group schemes
    over $R$ with special fiber $G$,
    then after possibly replacing $R$ by a finite extension of DVRs,
    $\mathcal{G}_1$ and $\mathcal{G}_2$ are isomorphic as
    group schemes over $R$.
    \end{enumerate}
\end{enumerate}
\end{Proposition}

\begin{proof}
This is discussed in \cite[Section 2]{LMM}, where (1) 
is \cite[Lemma 2.1]{LMM}, where (2) directly follows from (1), 
and where the assertions in (3) follow from \cite[Proposition 2.4]{LMM}.
\end{proof}

\begin{Definition}
The finite group $G_{\mathrm{abs}}$ associated to $G$ is called the
\emph{abstract group} associated to $G$.
The distinguished lift $\mathcal{G}_{\mathrm{can}}\to\Spec W(k)$ 
is called the \emph{canonical lift} of $G$.
\end{Definition}

After these preparations, we can state and prove the rigidity result
of this section for complete DVRs.

\begin{Proposition}
\label{prop: G over local complete DVR}
    Let $(R,\idealm)$ be a complete DVR,
    let $s\in S:=\Spec R$ be a point, and let $\kappa(s)$ be its 
    residue field.
    \begin{enumerate}
        \item If $G_i\to S$, $i=1,2$ are two finite flat linearly
        reductive group schemes, such that 
        $G_{1,s}\cong G_{2,s}$ as group schemes
        over $\kappa(s)$, then after possibly
        replacing $R$ by a finite extension
        of DVRs, 
        $G_1\cong G_2$ as group schemes over $S$.
        \item If $G\to S$ is a finite flat and linearly
        reductive group scheme and
        if $\rho_i:G\to\GL_{n,S}$, $i=1,2$ are 
        two linear representation
        of group schemes over $S$, such that
        $\rho_{1,s}\cong\rho_{2,s}$, that is, they are equivalent 
        as representations of $G_s$ over $\kappa(s)$,
        then $\rho_1\cong\rho_2$ as representations
        of $G$ over $S$.
        \item If $G_i\to S$, 
        $i=1,2$ are two finite flat linearly
        reductive subgroup schemes of $\GL_{n,S}$, such that 
        $G_{1,s}$ and $ G_{2,s}$ 
        are conjugate as subgroup schemes of $\GL_{n,\kappa(s)}$,
        then after possibly
        replacing $R$ by a finite extension
        of DVRs, 
        $G_1$ and $G_2$ are conjugate as subgroup schemes
        of $\GL_{n,S}$.
    \end{enumerate}
   Moreover, in each of the three statements if
   $s\in S$ is the closed point, then - after possibly extending 
   $R$ by a finite extension of DVRs - we can extend the
   given isomorphisms from $s$ to $S$.
\end{Proposition}

\begin{proof}
Let $k$ be the residue field of $R$ and let $K$ be
its field of fractions.

(1) First, assume that $R$ is of equal characteristic zero.
Then, $G_1$ and $G_2$ are both
finite and \'etale over $S$.
After possibly replacing $R$ by a finite extension,
both of them become the constant group schemes over $S$
that are associated to the finite groups 
$G_1(S)$ and $G_2(S)$, respectively.
The assumption then implies that $G_1(S)\cong G_2(S)$
and one can extend
the given isomorphism.

Second, assume that $R$ is of equal positive characteristic.
Then, we have $R=k[[t]]$ by the Cohen structure theorem.
It then follows from \cite[Proposition 2.4]{LMM}
that after possibly replacing $R$ by a finite extension,
we have that $G_i\cong G_{i,\idealm}\times_{\Spec k}S$
and the assertion follows in this case.

Third, if $\cha k=p>0$ and $s=\idealm\in S$ is the closed point, 
then this is \cite[Proposition 2.4]{LMM}.

Fourth, assume that $\cha K=0$, $\cha k=p>0$, and $s=\eta\in S$
is the generic point.
Let $\overline{K}$ be an algebraic closure of $K$.
Then the abstract groups 
$G_{1,\mathrm{abs}}:=G_1(\overline{K})$ and
$G_{2,\mathrm{abs}}:=G_2(\overline{K})$ 
associated to $G_1$ and $G_2$ are isomorphic.
On the other hand, $G_i\to\Spec R$, $i=1,2$ are lifts 
of their special fibers $G_{i,\idealm}\to\Spec k$
and thus, the abstract group associated to 
$G_{i,\idealm}$ is $G_{i,\mathrm{abs}}$ (by definition).
This implies that $G_{1,\idealm}$ and $G_{2,\idealm}$
become isomorphic over some finite extension of $k$.
Thus, after possibly replacing $R$ by some finite extension,
we may assume that $G_{1,\idealm}$ and $G_{2,\idealm}$
are isomorphic as group schemes over $k$.
The assertion then follows from the already established
third case.

(2) 
First, assume that $s\in S$ is the closed point.
Then, the proof of \cite[Proposition 2.9]{LMM}
shows that any two lifts of a representation of $G_s$
are isomorphic, and the assertion follows
in this case.

Second, assume that $s=\eta\in S$ is the generic point.
We have that the representation 
$\rho_{i,\eta}$ of $G_{i,\eta}$ is the lift
of the representation 
$\rho_{i,\idealm}$ of $G_{i,\idealm}$ for $i=1,2$.
Thus, \cite[Proposition 2.9.(3)]{LMM} shows
that $\rho_{1,\idealm}$ and $\rho_{2,\idealm}$ 
are isomorphic as representations of $G_\idealm$
and \cite[Proposition 2.9.(2)]{LMM}
shows that  $\rho_1$ and $\rho_2$ are isomorphic
as representations of $G$ over $S$.

(3)
Since $G_{1,s}$ and $G_{2,s}$ are conjugate inside $\GL_{n,\kappa(s)}$,
they are abstractly isomorphic as group schemes over $\kappa(s)$.
By (1), we may thus assume that $G:=G_1\cong G_2$ after possibly
replacing $R$ by a finite extension.
Thus, we may assume that we have two closed
embeddings $\varphi_i:G\to\GL_{n,S}$, $i=1,2$ of group schemes over $S$,
such that $\varphi_{1,s}$ is conjugate to $\varphi_{2,s}$.
Arguing as in the proof of assertion (1) of 
\cite[Proposition 2.9]{LMM}, it follows that 
such a conjugation can be spread out to $R$ if $s\in S$ is the
generic point or lifted to $R$ if $s\in S$ is the closed point.
\end{proof}

\begin{Example}
 The assumption on linear reductivity in 
 Proposition \ref{prop: G over local complete DVR}.(1)
 cannot be dropped, as the following examples show:
 Let $p$ be a prime, let $R$ be equal to $\ZZ_p[\sqrt[p]{p}]$ or $\FF_p[[t]]$,
 and set $S:=\Spec R$.
 Set $\varpi:=\sqrt[p]{p}$ in the first case and 
 $\varpi:=t$ in the second case.
 Thus, $\varpi\in R$ is a uniformizer.
 
 To give a finite flat group scheme $G\to S$
 of length $p$ is equivalent to giving a tuple 
 $(a,b)\in R\times R$ with $ab=p$ modulo the equivalence relation
 $(a,b)\sim(a',b')$ if and only if $\exists u\in R^\times$
 with $a'=u^{p-1}a$, $b'=u^{1-p}b$.
 These are the \emph{Oort--Tate parameters} of $G$
 and we refer to \cite{OortTate}, especially to the top of page 2,
 for details.
 \begin{enumerate}
 \item Let $G_1\to S$ be the finite flat group scheme of length $p$
 with Oort--Tate parameters $(a,b)=(\varpi^{p-1},\varpi)$ and $(t,0)$,
 respectively.
 Its generic fiber is isomorphic to $\ZZ/p\ZZ$ and 
 its special fiber is isomorphic to $\balpha_p$.
 In particular, it has the same generic fiber as the 
 constant group scheme $(\ZZ/p\ZZ)_S\to S$,
 which has Oort--Tate parameters $(1,p)$,
 but it is not isomorphic to it.
 They also do not become isomorphic after replacing $S$ by
 a faithfully flat extension.
 \item Let $G_2\to S$ be the Cartier dual of $G_1$, that is, 
 the finite flat group scheme of length $p$
 with Oort--Tate parameters $(a,b)=(\varpi,\varpi^{p-1})$ 
 and $(0,t)$, respectively.
 Here, the generic fiber is isomorphic to $\bmu_p$ and the 
 special fiber is isomorphic to $\balpha_p$.
 This group scheme has the same generic fiber as 
 $\bmu_{p,S}\to S$, which has Oort--Tate parameters
 $(p,1)$, but it is not isomorphic to it, 
 nor after replacing $S$ by a faithfully flat extension.
\end{enumerate}
Note that in the second example, the generic fibers 
of both group schemes are linearly reductive over the
generic point of $S$ and 
$\bmu_{p,S}\to S$ is even linearly reductive over $S$, 
whereas $G_2\to S$ is not linearly reductive over $S$.
\end{Example}

\subsection{Dedekind schemes}
Let us recall that a \emph{Dedekind scheme} is a Noetherian integral 
scheme of dimension 1, all of whose local rings are regular.
Before proceeding, we recall the following result,
which should be well-known. 

\begin{Lemma}
\label{lem: aut and isom}
    Let $S$ be a scheme.
    If $G_i\to S$, $i=1,2$ are two finite and flat 
    group schemes, then the functor
    $$\begin{array}{ccc}
   (\mathrm{Sch}/S) &\to& (\mathrm{Sets}) \\
   T &\mapsto& \mathrm{Isom}(G_{1,T},G_{2,T}),
   \end{array}
    $$
    where $\mathrm{Isom}(G_{1,T},G_{2,T})$ is the set
    of isomorphisms of group schemes over $T$, is representable
    by a scheme $\Isom(G_1,G_2)\to S$, which is
    affine and of finite type over $S$.

    In the special case $G:=G_1=G_2$, this shows that 
    the automorphism functor is represented by a group scheme
    $\AutScheme_{G}\to S$, which is affine
    and of finite type over $S$.
\end{Lemma}

\begin{proof}
First, we show that if $G\to S$ is a finite and flat group
scheme, then $\AutScheme_G\to S$ exists as an affine
group scheme of finite type over $S$.
The question is local on $S$ and thus, we may assume $S$
to be affine and we may also assume that $\OO_G$ is a 
free $\OO_S$-module of some finite rank $n$.
But then, the functor $T\mapsto\Aut(G_T)$ 
is a closed subfunctor
of the functor $T\mapsto\Aut(\OO_{G,T})$,
which is representable by the scheme $\GL_{n,S}$, which is
affine and of finite type over $S$.
Thus, $\AutScheme_G$ can be realized as a closed subgroup 
scheme of $\GL_{n,S}$ and thus, it also affine and of finite
type over $S$.

Second, we show the first claim of the lemma:
The functor is a right $\AutScheme_{G_1}$-torsor 
in the fppf topology.
Thus, it is representable by a scheme
\cite[Theorem III.4.3]{Milne}, which is also affine and of finite type over $S$
because $\AutScheme_{G_1}$ has these properties.
\end{proof}

\begin{Remark}
Probably, more can be said in the case of finite flat linearly 
reductive group schemes.
For example, we will see in forthcoming work that the automorphism group
scheme of a finite linearly reductive group scheme over a field is not 
only affine, but also finite.
This is in contrast to $\AutScheme_{\balpha_p}\cong\GG_m$, which is an
example of a finite group scheme whose automorphism group scheme is
positive dimensional.
\end{Remark}

\begin{Theorem}
\label{thm: G over Dedekind}
Let $S$ be an excellent Dedekind scheme,
let $s\in S$ be a point, and let $\kappa(s)$ be its 
residue field.
\begin{enumerate}
\item Let $G_i\to S$ be two finite flat linearly
 reductive group schemes and $G_{1,s}\cong G_{2,s}$ as group schemes
 over $\kappa(s)$.
 Then, there exists an fppf covering
 $\{S_j'\}_{j\in J}$ of $S$, such that 
 $G_{1,S_j'}\cong G_{2,S_j'}$ as group schemes over $S_j'$
 for all $j\in J$.
\item Let $G\to\Spec S$ be a finite flat and linearly
 reductive group scheme over $S$ and let 
 $\rho_i:G\to\GL_{n,S}$, $i=1,2$ be
 two linear representations of group schemes over $S$, 
 such that $\rho_{1,s}\cong\rho_{2,s}$
 as representations of $G_s$ over $\kappa(s)$.
 Then, there exists an fppf covering
 $\{S_j'\}_{j\in J}$ of $S$, such that 
 $\rho_{1,S_j'}\cong\rho_{2,S_j'}$ as representations
 of $G_{S_j'}$ over $S_j'$ for all $j\in J$.
\item Let $G_i\to S$, $i=1,2$ be two finite flat linearly
 reductive subgroup schemes of $\GL_{n,S}$, such that 
 $G_{1,s}$ and $ G_{2,s}$ are conjugate as closed subgroup schemes 
 of $\GL_{n,\kappa(s)}$.
 Then, there exists an fppf covering $\{S_j'\}_{j\in J}$ of $S$, such that 
 $G_{1,S_j'}$ and $G_{2,S_j'}$ are conjugate as closed subgroup schemes
 of $\GL_{n,S_j'}$ for all $j\in J$.
\end{enumerate} 
\end{Theorem}

\begin{proof}
First, we show that it suffices to prove the assertions
in the affine case.
Because then, we can choose an affine neighborhood
of $s\in S$ and the assertion in the affine case provides
us with an fppf neibhorhood of the generic
point $\eta\in S$, in which the statement holds.
Given an arbitrary point $t\in S$, it is contained in
some affine neighborhood, which also contains
$\eta$ and again, the affine case provides us
with an fppf neighborhood of $t$ in which the statement
holds.
Thus, we may and will now assume that $S=\Spec R$
for some excellent Dedekind ring $R$.

(1) We first prove there is an fppf neighborhood $S'\to S$ 
of $s$ where the result holds. 
By \cite[Theorem 2.16]{AOV}, after replacing $S$ by a connected fppf cover, 
we may assume that the $G_i$, $i=1,2$ both are well-split, 
that is, $G_i=\Delta_i\rtimes H_i$, where $\Delta_i$ is 
diagonalizable and $H_i$ is constant and tame. 
Then after looking further \'etale locally, we can assume $\Delta_i=G_i^\circ$ 
is the connected component of the identity. 
Thus, our isomorphism $G_{1,s}\cong G_{2,s}$ implies $H_1\cong H_2$ and 
$\Delta_{1,s}\cong \Delta_{2,s}$.
However since the Cartier duals of the $\Delta_i$ are constant, 
this implies $\Delta_1\cong \Delta_2$. 
Thus, the $G_i$ both of the form $\Delta\rtimes H$ with possibly different 
semi-direct product structures, that is, 
both are defined by maps $\varphi_i\colon H\to\AutScheme_\Delta$. 
Since the $\varphi_i$ are maps of constant group schemes and they are equal over $s$, they must be globally equal.

We have thus produced our fppf neighborhood $f\colon S'\to S$.
The image of $f$ contains the generic point because it is an open map. 
Thus, if $\overline{K}$ is an algebraic closure of the field of
fractions $K$ of $R$, then $G_{1,\overline{K}}\cong G_{2,\overline{K}}$
as group schemes over $\overline{K}$.
In fact, there exists a finite field extension $K\subseteq L$, such
that $G_{1,L}\cong G_{2,L}$ as group schemes over $L$.

Now, if $t\in S$ is a closed point, then let $\idealn\subset R$ be
the maximal ideal corresponding to $t\in S$,
let $R_\idealn$ the localization, and let $\widehat{R}_{\idealn}$
be the $\idealn$-adic completion.
Since $G_{1,\eta}\cong G_{2,\eta}$, it 
follows from Proposition \ref{prop: G over local complete DVR}
that $G_{1,t}$ and $G_{2,t}$ become isomorphic over some finite field
extension of the residue field $\kappa(t)=R/\idealn$.
Replacing $R$ by some finite flat extension, we may assume that 
$G_{1,t}$ and $G_{2,t}$ are isomorphic over $\kappa(t)$
and arguing as before, we find an fppf neighborhood 
$T'\to S$ of $t$, such that $G_{1,T}\cong G_{2,T}$ as group
schemes over $T$.

(2) 
Again, if $s=\eta\in S$ is the generic point, then there
exists a Zariski open $\emptyset\neq U\subseteq S$,
such that $\rho_{1,U}\cong\rho_{2,U}$ as representations
of $G_U$ over $U$.

If $s\in S$ is a closed point, let $\idealm\subset R$
be the corresponding maximal ideal, let
$R\to R_\idealm\to\widehat{R}_\idealm$ be the localization at
$\idealm$ followed by the $\idealm$-adic completion,
and let $R_\idealm\subseteq R^h\subseteq \widehat{R}_{\idealm}$
be the Henselization.
By assumption and Proposition \ref{prop: G over local complete DVR},
we have that $\rho_1$ and $\rho_2$ are isomorphic over
$\Spec\widehat{R}$.
By approximation, $\rho_1$ and $\rho_2$ are isomorphic over $R^h$,
see \cite[Corollary (2.4)]{ArtinApprox}.
Thus, there exists an \'etale neighborhood $S'\to S$
of $s\in S$, such that $\rho_{1,S'}\cong \rho_{2,S'}$
as representations of $G_{S'}$ over $S'$.

In any case, there exists an \'etale neighborhood of the generic point
over which $\rho_1$ and $\rho_2$ are isomorphic.
Thus, if $\overline{K}$ is an algebraic closure of the field of
fractions $K$ of $R$, then $\rho_{1,\overline{K}}\cong \rho_{2,\overline{K}}$
representations of $G_{\overline{K}}$ over $\overline{K}$.
In fact, there exists a finite field extension $K\subseteq L$, such
that $\rho_{1,L}\cong \rho_{2,L}$ as representations of $G_L$
over $L$.

Now, if $t\in S$ is a closed point, then let $\idealn\subset R$ be
the corresponding maximal ideal,
let $R_\idealn$ the localization, and let $\widehat{R}_{\idealn}$
be the $\idealn$-adic completion.
Since $G_{1,\eta}\cong G_{2,\eta}$, it 
follows from Proposition \ref{prop: G over local complete DVR}
that $G_{1,t}$ and $G_{2,t}$ become isomorphic over some finite field
extension of the residue field $\kappa(t)=R/\idealn$.
Replacing $R$ by some finite flat extension,
we may assume that 
$\rho_{1,t}$ and $\rho_{2,t}$ are isomorphic over $\kappa(t)$
and arguing as before, we find an fppf neighborhood 
$T\to S$ of $t$, such that $\rho_{1,T}\cong \rho_{2,T}$ as 
representations of $G_T$.

(3) Is shown as (2) and we leave it to the reader.
\end{proof}

\section{Locally conjugate subgroup schemes}
\label{sec: twisted forms}

In this section, we study closed subgroup
schemes of $\GL_{n,S}$ over some base scheme $S$ that are locally
in the fppf topology conjugate to each other.
We give a cohomological description,
which extends \cite[Theorem 1.3 and Theorem 1.4]{LS2},
where we established this in the case where $S$ is the spectrum
of a field.

\subsection{Locally conjugate subgroup schemes and normalizers}
Let $S$ be a base scheme and let $G_i\to S$, $i=1,2$ 
and $X\to S$ be group schemes that are finitely presented, affine,
and flat group over $S$.
Assume that both, $G_1$ and $G_2$ are closed subgroup schemes of $X$
over $S$, which we assume to be Noetherian and separated by our
standing assumption.

\begin{Definition}
For a Grothendieck topology $\mathcal{T}$  on $S$, 
we will say that
$G_1\subseteq X$ is \emph{locally in the $\mathcal{T}$-topology  
conjugate} to $G_2\subseteq X$ if
there exists a $\mathcal{T}$-cover $\{S'_j\}_{j\in J}$ of $S$, such that
$G_{1,S'_j}\subseteq X_{S'_j}$ is $X(S'_j)$-conjugate to
$G_{2,S'_j}\subseteq X_{S'_j}$ for all $j\in J$.
\end{Definition}

In order to state Theorem \ref{thm: brian}, we recall the notion
of a normalizer scheme:
for a base scheme $S$, a group scheme $X\to S$ and a closed subgroup scheme 
$G\subseteq X$ over $S$, we consider the functor 
$$\begin{array}{ccc}
   (\mathrm{Sch}/S) &\to& (\mathrm{Groups}) \\
   T &\mapsto& \{x\in X(T) \,:\, x\, G_T\, x^{-1}\,\subseteq\,X_T\}.
   \end{array}
$$
In our setting where $S$ is Noetherian and $G$ and $X$ 
are affine and of finite type over $S$, 
this functor is representable by a closed subgroup scheme 
$\Norm_X(G)$ of $X$ over $S$, see, 
for example, \cite[Proposition 2.1.2]{Conrad}.

\subsection{Classification via torsors and cohomology}
In the special case where $S$ is the spectrum of a field,
we already obtained the following result 
in \cite[Theorem 1.4]{LS2} and \cite[Corollary 2.1]{LS2}.

\begin{Theorem}
\label{thm: brian}
 Let $S$ be a Noetherian separated 
 scheme and let $G$ and $X$ be group schemes 
 that are affine, flat, and of finite type over $S$. 
 Assume that $G$ is a closed subgroup scheme of $X$ over $S$.
 Then, there exists a bijection between the following pointed sets:
 \begin{enumerate}
 \item Subgroup schemes of $X$ that are locally in the fppf topology 
 conjugate to $G\subset X$, up to $X(S)$-conjugacy.
 \item $\Norm_X(G)$-torsors $E\to S$ in the fppf topology 
 that admit an equivariant embedding into $X$.
 \item Elements of the pointed cohomology set
 $$
 \ker\left( \Hfl{1}(S,\Norm_X(G)) \,\to\, \Hfl{1}(S,X) \right).
 $$
 \end{enumerate}
\end{Theorem}

\begin{proof}
We extend the proof \cite[Theorem 1.4]{LS2} to our setting.

$(1)\Rightarrow(2)$: 
Let $H\subset X$ be a closed subgroup scheme that becomes
conjugate to $G\subset X$ over some fppf cover $\{S_j'\}_{j\in J}$ of $S$.
By \cite[Proposition 2.1.2]{Conrad}, the functor
$$
\begin{array}{ccc}
(\mathrm{Schemes}/S) &\to& (\mathrm{Sets})\\
S &\mapsto& \{x\in X(S)| x\cdot H_S\cdot x^{-1}=G_S\}
\end{array}
$$
is representable by a closed subscheme $E_H$ of $X$. 
It is also a left $\Norm_X(G)$-torsor and since we have
$E_H(S_j')\neq\emptyset$ for all $j\in J$,
this $\Norm_X(G)$-torsor  splits over $\{S_j'\}_{j\in J}$.
Furthermore, the embedding $E_H\subset X$ 
is $\Norm_X(G)$-equivariant.

$(2)\Rightarrow(1)$: 
Let $E$ be a left $\Norm_X(G)$-torsor over $S$ that splits over an fppf
cover $\{S_j'\}_{j\in J}$ of $S$.
Since $E(S_j')\neq\emptyset$ for all $j\in J$, we may pick elements
$t_j\in E(S_j')$.
We note that $E(S_j')=\Norm_X(G)\cdot S_j'$, which shows that
$t_j^{-1}\cdot G_{S'_j}\cdot t_j\subset X_{S_j'}$ 
is independent of the choice of element $t_j$.
Then, we define
$$
H_E \,:=\, 
(G\times_SE)/\Norm_X(G) \,\to\, 
(X\times_SE)/\Norm_X(G)\,=\,X,
$$
where the action of $\Norm_X(G)$ on $G\times_S E$ and $X\times_SE$
is diagonal via conjugation on the first factor and the 
torsor structure on the second factor.
This way, we obtain a subgroup scheme $H_E\subseteq X$ over $S$ 
that becomes $t_j^{-1}\cdot G_{S_j'}\cdot t_j\subseteq X_{S_j'}$ 
over $S_j'$.


It is straightforward to check that the constructions
$H\mapsto E_H$ and $E\mapsto H_E$ are inverse to each other
and we obtain a bijection.
In both cases, there is a distinguished element, namely
the $X(S)$-conjugacy class of $G\subset X$ in (1)
and the trivial $\Norm_X(G)$-torsor in (2).
Our bijection sends these distinguished elements to
each other and
we obtain a bijection of pointed sets.

$(2) \Leftrightarrow (3)$:
It suffices to note that a left $\Norm_X(G)$-torsor $E$ over $S$
admits an equivariant closed immersion 
into $X$ if and only if its pushout to $X$ is trivial
and thus, if and only if it lies in the kernel
of the map
$\Hfl{1}(S,\Norm_X(G))\to\Hfl{1}(S,X)$.
\end{proof}

\begin{Remark} 
 With notations and assumptions as in Theorem \ref{thm: brian},
 we note that a sheaf of $\Norm_X(G)$-torsors with respect to the fppf
 topology $S$ is representable by a scheme, see \cite[Theorem III.4.3]{Milne},
 that is, we do not have to worry about representability.
\end{Remark}

\section{The key observation}
\label{sec: key}

In this section, we study the Klein sets $\Klein(n,S)$, that is,
finite flat linearly reductive subgroup schemes $G\subset\SL_{2,S}$ of length 
$n$ over $S$ (see Section \ref{subsec: Klein set}),
in the special case where $S$ is an excellent Dedekind scheme.
The key result is that if $S$ has a point of residue characteristic $2$,
then $G\subset\SL_{2,S}$ is locally in the fppf topology conjugate to  
the standard embedding $\bmu_n\subset\SL_{2,S}$.
Thus, we only have one  infinite series and its conjugates in the fppf topology as opposed to
two infinite series and three exceptional
cases -- as over the complex numbers and as recalled 
in Section \ref{sec: Klein classification}.

For any base scheme $S$ and any integer $n\geq1$,
we have the finite flat linearly reductive
group scheme $\bmu_{n,S}\to S$.
We have an isomorphism (Cartier duality)
$$
\Hom\left(\bmu_{n,S},\GG_{m,S}\right) 
\,\cong\, (\ZZ/n\ZZ)_S
$$
of constant group schemes over $S$.
If $S$ is connected, then 
we choose a generator $\psi$ of this
cyclic and constant group scheme and define
$$
\rho\,:=\,\psi\oplus\psi^\vee \quad:\quad
\bmu_n\quad\to\quad\SL_{2,S}\,\subset\,\GL_{2,S}
$$
to be the \emph{standard embedding} of $\bmu_{n,S}$ 
into $\SL_{2,S}$.
If we pick another generator, then 
this changes $\rho$ by pre-composing
$\bmu_{n,S}$ with an automorphism
of $\bmu_{n,S}$.
In particular, the image of $\rho$
does not depend on the choice of $\psi$.

\begin{Theorem}
\label{thm: key}
  Let $R$ be an excellent Dedekind ring, let $n\geq2$ be an integer,
  and let $G\in\Klein(n,R)$.
  Assume that there exists a point of $S$ of residue characteristic $2$.
  Then, $G\subset\SL_{2,R}$ is locally in the fppf topology
  conjugate to the standard embedding 
  $\bmu_{n,R}\subset\SL_{2,R}$.
\end{Theorem}

\begin{proof}
Let $K$ be the field of fractions of $R$,
let $s\in S$ be a point whose residue field $k:=\kappa(s)$
has characteristic $2$, and let $k\subseteq\overline{k}$
be an algebraic closure.
By Theorem \ref{thm: Klein}, we know that 
$(G_s)_{\overline{k}}\subset\SL_{2,\overline{k}}$ is conjugate
to the standard embedding of
$\bmu_{n,\overline{k}}\subset\SL_{2,\overline{k}}$.
Thus, after possibly replacing $R$ by its integral closure
in a suitable finite extension of $K$, we may assume
that $G_s\subset\SL_{2,s}$ is conjugate to 
$\bmu_{n,s}\subset\SL_{2,s}$.
The assertion now follows from 
Theorem \ref{thm: G over Dedekind}.
\end{proof}

Combined with Theorem \ref{thm: brian},
we obtain the following
cohomological description of the Klein set 
$\Klein(n,R)$ introduced in
Section \ref{subsec: Klein set}.

\begin{Corollary}
\label{cor: key}
 With notations and assumptions of Theorem \ref{thm: key},
 there exists a bijection of pointed sets
 $$
\Klein(n,R) \quad\cong\quad \ker\left(
\Hfl{1}(R,\Norm_{\GL_{2,R}}(\bmu_{n,R})) 
\,\to\, \Hfl{1}(R,\GL_{2,R})\right),
 $$
 where the distinguished elements on both sides
 correspond to the class of the standard
 embedding $\bmu_{n,R}\subset\SL_{2,R}$.\qed
\end{Corollary}

\section{Twisted forms of $\mu_n$ inside $\mathrm{SL}_2$}
\label{sec: twisted forms of mun}

In this section, we start from an arbitrary base scheme $S$
and the standard embedding (see Section \ref{sec: key})
$$
 \psi\oplus\psi^\vee\,:\,\bmu_{n,S} \,\to\,\SL_{2,S} \,\subset\,\GL_{2,S}.
$$
By Theorem \ref{thm: brian}, the $\GL_2(S)$-conjugacy classes of
closed subschemes of $\GL_{2,S}$ that are fppf locally 
conjugate to this embedding are in bijection to the pointed set
$$
  \ker\left(\, \Hfl{1}(S,\Norm_{\GL_{2,S}}(\bmu_n)) \,\to\, \Hfl{1}(S,\GL_{2,S}) \,\right),
$$
where we use the normalizer scheme of $\bmu_{n,S}$ inside
$\GL_{2,S}$.
Of course, the distinguished element of this pointed set corresponds
to the standard embedding $\bmu_{n,S}\subset\SL_{2,S}$.

First, we will compute this normalizer, then
we will compute its first flat cohomology set, and finally, 
we will describe the associated twisted subgroup schemes
of $\GL_{2,S}$.
In the special case where $S$ is the spectrum of a field, 
we have already carried out this analysis in 
\cite[Proposition 4.1]{LS2}, \cite[Theorem 5.1]{LS2}, and \cite[Section 6.1]{LS2}.
A new feature when working over a general base schemes $S$
is that the Picard group of $S$ will play a role, which leads
to embeddings $\bmu_{n,S}\subset\SL_{2,S}\subset\GL_{2,S}$ 
that are Zariski locally conjugate to
the standard embedding.

\subsection{The normalizer}
In order to apply Theorem \ref{thm: brian}, we start with computing the 
normalizers.

\begin{Proposition}
\label{prop: normalizers}
 Let $S$ be a scheme and let $n\geq2$ be an integer.
 \begin{enumerate}
 \item There is an isomorphism of group schemes over $S$
 $$
  \AutScheme_{\bmu_{n,S}} \,\cong\, (\ZZ/n\ZZ)^\times_S.
  $$
  In particular, these are finite and \'etale group schemes over $S$.
 \item Let  $\bmu_{n,S}\subset\SL_{2,S}$ be the standard embedding.
 \begin{enumerate}
 \item If $n=2$, then 
 $$
 \begin{array}{lclc}
  \Norm_{\SL_{2,S}}(\bmu_{2,S}) &=& \SL_{2,S} \\
  \Norm_{\GL_{2,S}}(\bmu_{2,S}) &=& \GL_{2,S}&.
 \end{array}
 $$
     \item If $n\geq3$, then there exist split exact sequences 
 $$
    \begin{array}{ccccccccc}
     1&\to&\GG_{m,S}&\to&\Norm_{\SL_{2,S}}(\bmu_{n,S})&\to&(\ZZ/2\ZZ)_S&\to&1 \\
     1&\to&\GG_{m,S}^2&\to&\Norm_{\GL_{2,S}}(\bmu_{n,S})&\to&(\ZZ/2\ZZ)_S&\to&1 
     \end{array}
 $$
 of group schemes over $S$, where $\GG_{m,S}\subset\SL_{2,S}$ 
 and $\GG_{m,S}^2\subset\GL_{2,S}$ are the standard tori.
\item If $n\geq3$ and if $\GG_{m,S}^2\subset\GL_{2,S}$ is the standard torus, then
$$
 \Norm_{\GL_{2,S}}(\bmu_{n,S})\,=\, \Norm_{\GL_{2,S}}(\GG_{m,S}^2)\,\subset\,\GL_{2,S}.
$$
In particular, the normalizer is independent of $n$ if $n\geq3$.
\end{enumerate}
 \item Set 
 $$
   K \,:=\,\GG_{m,S}\cap\ker(\det) \,\subset\,\SL_{2,S}\,\subset\,\GL_{2,S},
 $$
 which is isomorphic to $\GG_{m,S}$.
 If $n\geq3$, then conjugation induces homomorphisms 
 $$
 \begin{array}{cclcl}
 \gamma &:& \Norm_{\GL_{2,S}}(\GG_{m,S}^2) &\to& \AutScheme_{K}\\
 \gamma_n &:& \Norm_{\GL_{2,S}}(\bmu_{n,S}) &\to& \AutScheme_{\bmu_{n,S}},
 \end{array}
 $$
 of group schemes over $S$,
 whose image is generated by the automorphism (`inversion')
 $x\mapsto x^{-1}$ of $K$ and $\bmu_{n,S}$,
 respectively.
 In particular, the images of $\gamma$ and $\gamma_n$ are
 isomorphic to $(\ZZ/2\ZZ)_S$.
\end{enumerate}
\end{Proposition}

\begin{proof}
The assertions in (1) and (2) are well-known and we refer to
\cite[Proposition 4.1]{LS2} and \cite[Section 6.1]{LS2}
for a discussion already with a view towards classifying
twisted subgroup schemes of the standard embedding
$\bmu_n\subset\SL_2$. 
The splitting of the exact sequences in (2).(b) is obtained
by means of the matrix
\begin{equation}
\label{eq: splitting matrix}
\left(
\begin{array}{cc}
0 & 1 \\
-1 & 0 
\end{array}
\right)
\,\in\, \Norm_{\SL_{2,S}}(\bmu_{n,S})
\,\subset\,\Norm_{\GL_{2,S}}(\bmu_{n,S})
\end{equation}

(3) This is straightforward computation that 
we leave to the reader.
\end{proof}

\begin{Corollary}
 \label{cor: mu2 case}
  Let $S$ be a scheme and
  let $G\subset\GL_{2,S}$ be fppf locally conjugate
  to the standard embedding $\bmu_{2,S}\subset\SL_{2,S}\subset\GL_{2,S}$.
  Then, $G$ is $\SL_2(S)$-conjugate and $\GL_2(S)$-conjugate 
  to $\bmu_{2,S}$, in $\GL_{2,S}$.
\end{Corollary}

\begin{proof}
By Theorem \ref{thm: brian}, $G$ is determined 
up to $X(S)$-conjugacy by its cohomology class
in the kernel of $\Hfl{1}(S,\Norm_X(\bmu_{2,S}))\to\Hfl{1}(S,X)$, where
$X=\SL_{2,S}$ or $X=\GL_{2,S}$.
Since we have $\Norm_X(\bmu_{2,S})=X$ in both cases
by Proposition \ref{prop: normalizers}.(2).(a),
this kernel is trivial in both cases and the claim follows.
\end{proof}

We keep the notations and assumptions from Proposition \ref{prop: normalizers} and 
assume moreover that $n\geq3$.
It follows from Proposition \ref{prop: normalizers}.(3) that
conjugation induces a map of pointed cohomology sets
$$
  \gamma_{n,\ast}\,:\,\Hfl{1}(S,\Norm_{\GL_{2,S}}(\bmu_{n,S})) 
  \,\to\,
  \Hfl{1}(S,\AutScheme_{\bmu_{n,S}})
$$
and that this map factors through 
$\Hfl{1}(S,(\ZZ/2\ZZ)_S)$.
From Theorem \ref{thm: brian} and the previous
discussion we obtain the following interpretation of
this map.

\begin{Corollary}
\label{cor: quadratic twist}
  Let $S$ be a scheme and 
  let $G\subset\GL_{2,S}$ be fppf locally conjugate
  to the standard embedding $\bmu_{n,S}\subset\SL_{2,S}\subset\GL_{2,S}$
  with $n\geq3$.
  Then,
  \begin{enumerate}
      \item $\gamma_{n,\ast}$ maps the class of $G$ considered as a 
      twisted subgroup scheme of $\GL_{2,S}$ 
      (as explained in Theorem \ref{thm: brian})
      to the class of $G$ considered as a twisted form of $\bmu_{n,S}$.
      \item $G$ is a quadratic twist of $\bmu_{n,S}$ with respect to
      inversion.\qed
  \end{enumerate}
\end{Corollary}

\begin{Corollary}
\label{cor: independence of n}
  If $S$ is a scheme and if $n\geq3$ is an integer,
  then there exists a bijection between the following sets:
  \begin{enumerate}
    \item Subgroup schemes $T\subset\GL_{2,S}$ that are fppf locally conjugate to
    the standard torus
    $\GG_{m,S}^2\subset\GL_{2,S}$.
    \item Subgroup schemes $G\subset\GL_{2,S}$ that are fppf locally conjugate to 
    the standard embedding
    $\bmu_{n,S}\subset\GL_{2,S}$.
  \end{enumerate}
\end{Corollary}

\begin{proof}
By Theorem \ref{thm: brian} and by Proposition \ref{prop: normalizers}.(2).(c), 
both sets are classified by the kernel
of $\Hfl{1}(S,\Norm_{\GL_2}(\GG_m^2))\to\Hfl{1}(S,\GL_{2,S})$ 
and the statement follows.
\end{proof}

\begin{Remark}
\label{rem: independence of n}
In fact, we can be very explicit about this bijection:
if $T\subset\GL_{2,S}$ is fppf locally conjugate to $\GG_{m,S}^2\subset\GL_{2,S}$,
then, the intersection of $T[n]$ (kernel of multiplication by $n$) with
$\ker(\det)$ is a finite flat closed subgroup scheme $G\subset\GL_{2,S}$, which is
fppf locally conjugate to $\bmu_{n,S}\subset\SL_{2,S}\subset\GL_{2,S}$.
The corollary implies that the assignment 
$$
 T \quad \mapsto \quad T[n]\cap\ker(\det)
$$
defines a bijection between the sets (1) and (2).
\end{Remark}

\subsection{First cohomology of the normalizer}
By Proposition \ref{prop: normalizers}, we have an exact sequence
of pointed cohomology sets
\begin{equation}
\label{eq: split exact sequence}
1\,\to\,
\Hfl{1}(S,\GG_{m,S}^2)\,\to\,
\Hfl{1}(S,\Norm_{\GL_{2,S}}(\GG_{m,S}^2)) \,\to\,
\Hfl{1}(S,(\ZZ/2\ZZ)_S)\,\to\,1\,.
\end{equation}
To give an element
of $\Hfl{1}(S,\GG_{m,S}^2)$ is equivalent to giving a pair
$(\mathcal{L}_1,\mathcal{L}_2)$ of invertible $\OO_S$-modules
since $\Hfl{1}(S,\GG_{m,S})\cong\Pic(S)$.
Moreover, to give a class in $\Hfl{1}(S,(\ZZ/2\ZZ)_S)$ is equivalent
to giving a $(\ZZ/2\ZZ)$-torsor $S'\to S$, that is, an 
\'etale double cover.
In view of Theorem \ref{thm: brian}, we have to understand the kernel
of the map $\Hfl{1}(S,\Norm_{\GL_{2,S}}(\GG_{m,S}^2))\to\Hfl{1}(S,\GL_{2,S})$.
An element of $\Hfl{1}(S,\GL_{2,S})$ can be represented by a locally
free $\OO_S$-module of rank 2.

\begin{Lemma}
\label{lem: kernel cohomology}
 Let $S$ be a scheme.
 \begin{enumerate}
 \item The natural map
 $$
\Hfl{1}(S,\GG_{m,S}^2)\,\to\,
\Hfl{1}(S,\Norm_{\GL_{2,S}}(\GG_{m,S}^2)) \,\to\, \Hfl{1}(S,\GL_{2,S})
 $$
 assigns to the class of a pair $(\mathcal{L}_1,\mathcal{L}_2)$
 of invertible $\OO_S$-modules the class of the locally
 free $\OO_S$-module $\mathcal{L}_1\oplus\mathcal{L}_2$.
 \item Let $\sigma$ be the splitting of \eqref{eq: split exact sequence}
 defined by the matrix \eqref{eq: splitting matrix}.
 The map
  $$
\Hfl{1}(S,(\ZZ/2\ZZ)_S)\,\stackrel{\sigma}{\longrightarrow}\,
\Hfl{1}(S,\Norm_{\GL_{2,S}}(\GG_{m,S}^2)) \,\to\, \Hfl{1}(S,\GL_{2,S})
 $$
 assigns to the class of an \'etale double cover $f:S'\to S$
 the class of the locally free $\OO_S$-module $f_\ast\OO_{S'}$.
 \end{enumerate}
\end{Lemma}

\begin{proof}
(1) is clear.

(2) Let $f:S'\to S$ be an \'etale double cover, which is thus
Galois with group $\ZZ/2\ZZ$ and let $\tau\in\Aut(S'/S)$ 
be the non-trivial element.
Then, the image of its class in $\Hfl{1}(S,\GL_{2,S})$ can be
represented by the locally free $\OO_S$-module of rank $2$,
which is the $\OO_S$-submodule of $\ZZ/2\ZZ$
$$
  (\OO_{S'}\oplus\OO_{S'})^{\ZZ/2\ZZ} \,\subseteq\,
  \OO_{S'}\oplus\OO_{S'},
$$
where the non-trivial element of $\ZZ/2\ZZ$ acts as
$(a,b)\mapsto (\tau(b),\tau(a))$ on local sections.
This submodule is isomorphic to $\OO_{S'}$ considered
as an $\OO_S$-module via the embedding $a\mapsto (a,\tau(a))$.
\end{proof}

\subsection{Zariski local forms}
\label{subsec: Zariski}
We represent a class $\xi\in\Hfl{1}(S,\GG_{m,S}^2)$ by a pair
$(\mathcal{L}_1,\mathcal{L}_2)$ of invertible $\OO_S$-modules.
The associated $\Norm_{\GL_{2,S}}(\GG_{m,S}^2)$-torsor acts
trivially on $\GG_{m,S}^2$ because the conjugation-action 
of $\GG_{m,S}^2$ on itself is the identity and thus, 
the associated twisted form of $\GG_{m,S}^2$ is trivial.
The content of this section is that $\xi$
may nevertheless give rise to a non-standard embedding of $\GG_{m,S}^2$
into $\GL_{2,S}$ that is locally for the Zariski topology
conjugate to the standard torus.

Of course, if $S$ is the spectrum of a field or a local ring, 
then $\Pic(S)$ is trivial, which is why the
Zariski local forms of this section did not show up in
\cite{LS2}, where we worked over fields.

Next, for a locally free $\OO_S$-module $\mathcal{E}$ of finite rank $r$,
we have a sheaf of $\OO_S$-algebras 
$\shEnd(\mathcal{E})\cong\mathcal{E}\otimes\mathcal{E}^\vee$.
Its group of multiplicative units $\shEnd(\mathcal{E})^\times$
is an affine group scheme over $S$ that we denote by $\GL(\mathcal{E})\to S$.
Since $\mathcal{E}$ is isomorphic to $\OO_S^{\oplus r}$ locally in 
the Zariski topology, it follows that $\GL(\mathcal{E})\to S$ is 
a Zariski local form of $\GL_{r,S}\to S$.
If $r=1$, then $\GL(\mathcal{E})\cong\GL_{1,S}=\GG_{m,S}$.
Also, if $\mathcal{E}\cong\OO_S^{\oplus r}$, then 
$\GL(\mathcal{E})=\GL_{r,S}$.

\begin{Lemma}
\label{lem: Zariski twisted forms of torus}
 Let $S$ be a scheme and let $r\geq2$ be an integer.
 \begin{enumerate}
 \item Conjugation induces a homomorphism
 $\GL_{r,S}\to\AutScheme_{\GL_{r,S}}$ of group schemes over $S$
 and thus, an induced map of pointed cohomology sets
 $$
  \Hfl{1}(S,\GL_{r,S}) \,\to\, \Hfl{1}(S,\AutScheme_{\GL_{r,S}}).
 $$
 An element on the left can be represented by a locally
free $\OO_S$-module
 $\mathcal{E}$ of rank $r$ and its image on the right can
 be represented by $\GL(\mathcal{E})$, considered as 
 twisted form of $\GL_{r,S}$.
     \item
 The inclusion of the standard maximal torus
 $\GG_{m,S}^r\subset\GL_{r,S}$ induces a map of pointed cohomology sets
 \begin{equation}
     \label{eq: GL2 cohomology}
        \Hfl{1}(S,\GG_m^r)\,\to\,\Hfl{1}(S,\GL_{r,S}).
 \end{equation}
 A class on the left can be represented by a tuple 
 $(\mathcal{L}_1,...,\mathcal{L}_r)$ 
 of invertible $\OO_S$-modules
 and its image on the right is the class of the locally free $\OO_S$-module 
 $\mathcal{L}_1\oplus...\oplus\mathcal{L}_r$ of rank $r$.
 In particular, the kernel of this map is given by
 $$
\left\{(\mathcal{L}_1,...,\mathcal{L}_r)\in \Pic(S)^{\oplus r}\,|\, \mathcal{L}_1\oplus...\oplus\mathcal{L}_r\cong\OO_S^{\oplus r}\right\},
 $$
 where $\cong$ denotes isomorphism as an $\OO_S$-module.
  \item 
  Associated to an $r$-tuple
  $(\mathcal{L}_1,...,\mathcal{L}_r)\in\Pic(S)^{\oplus r}$
  of invertible $\OO_S$-modules one has a sequence
  $$
 \GG_{m,S}^r\,\cong\,\GL(\mathcal{L}_1)\times...\times\GL(\mathcal{L}_r)
 \,\to\, \GL(\mathcal{L}_1\oplus...,\oplus\mathcal{L}_r)
  $$
  of group schemes over $S$.
 \item Let $\GG_{m,S}^r\subset\GL_{2,S}$ be the standard torus.
 If the tuple $(\mathcal{L}_1,...,\mathcal{L}_r)\in\Pic(S)^{\oplus r}$
 lies in the kernel of \eqref{eq: GL2 cohomology},
 then the image of its class under the natural maps
 $$
\Hfl{1}(S,\GG_{m,S}^r)\,\to\,\Hfl{1}(S,\Norm_{\GL_{r,S}}(\GG_{m,S}^r))
\,\to\, \Hfl{1}(S,\GL_{r,S})
 $$
 in the pointed set in the middle corresponds
 to a twisted subgroup scheme of the standard torus inside $\GL_{r,S}$
 This twisted subgroup scheme can be represented by
 the maximal torus
 $$
 \GG_{m,S}^r\,\cong\,\GL(\mathcal{L}_1)\times...\times\GL(\mathcal{L}_r)
 \,\to\, \GL(\mathcal{L}_1\oplus...\oplus\mathcal{L}_r) \,\cong\, \GL_{r,S}
  $$
  from (3) and Zariski-locally, it is conjugate
  to the standard torus.
 \end{enumerate}
\end{Lemma}

\begin{proof}
(1) Let $\mathcal{E}$ be a locally free $\OO_S$-module of rank $r$
that we consider as a $\GL_{r,S}$-torsor and have 
an associated class in $\Hfl{1}(S,\GL_{r,S})$.
Then, the associated class in $\Hfl{1}(S,\AutScheme_{\GL_{r.S}})$
can be represented by the twist.
$$
\GL_{r,S}\wedge\mathcal{E}.
$$
Let us describe this twist in detail:
we choose a Zariski open cover $\{U_\alpha\}$, such that
$\mathcal{E}|_{U_\alpha}\cong\OO_{U_\alpha}$ and obtain a collection of
transition functions in $f_{\alpha,\beta}\in\GL(\OO_{U_\alpha\cap U_\beta}(U_\alpha\cap U_\beta))$
for each pair $\alpha,\beta$, whose associated 1-cocycle
gives the class of $\mathcal{E}$ in $\Hfl{1}(S,\GL_{n,S})$.
The twist $\GL_{r,S}\wedge\mathcal{E}$ 
is obtained by glueing for each pair of indices $\alpha,\beta$ the
group schemes $\GL(\OO_{\alpha})\to U_\alpha$ and $\GL(\OO_{\beta})\to U_\beta$ 
over $U_\alpha\cap U_\beta$ 
via conjugation by $f_{\alpha,\beta}$
(here, we use that the map $\GL_{r,S}\to\AutScheme_{\GL_{r,s}}$ is
defined by conjugation).
We leave it to the reader 
(one can use that conjugation by $f_{\alpha,\beta}$ gives the
transition functions of $\shEnd(\mathcal{E})$)
to check that this twist is isomorphic to the group scheme of units
of $\shEnd(\mathcal{E})$, that is, to $\GL(\mathcal{E})\to S$.

(2) is clear.

(3) follows from functoriality.

(4)
Given an $r$-tuple $(\mathcal{L}_1,...,\mathcal{L}_r)$ of invertible
$\OO_S$-modules, we set $\mathcal{E}:=\mathcal{L}_1\oplus...\oplus\mathcal{L}_r$.
As explained in the proof of Theorem \ref{thm: brian},
we obtain 
$$
\GG_{m,S}^r\wedge \mathcal{E} \,\to\, \GL_{r,S}\wedge\mathcal{E}
$$
where $-\wedge\mathcal{E}$ denotes twisting with the $\GG_{m}^r$-torsor
$\mathcal{E}$.
The term on the left is isomorphic to 
$\GL(\mathcal{L}_1)\times...\times\GL(\mathcal{L}_r)$
and as explained in (1), the term on the right is isomorphic 
to $\GL(\mathcal{E})$.
This is the map in (3) and as seen in the proof of 
Theorem \ref{thm: brian}, this is precisely how 
the map $\Hfl{1}(S,\Norm_{\GL_r}(\GG_m^r))\to\Hfl{1}(S,\GL_r)$
can be realized in terms of torsors.
If the $r$-tuple $(\mathcal{L}_1,...,\mathcal{L}_r)$ lies
in the kernel of this map, then $\GL(\mathcal{E})\cong\GL_{r,S}$.
\end{proof}

\begin{Remark}
In particular, when we specialize to the case $r=2$ we have the following. 
To pass from twisted subgroup schemes of $\GG_{m,S}^2\subset\GL_{2,S}$
to twisted forms of the standard embedding 
$\bmu_{n,S}\subset\SL_{2,S}\subset\GL_{2,S}$ with $n\geq3$
we use Corollary \ref{cor: independence of n} and
Remark \ref{rem: independence of n}:
Associated to a pair $(\mathcal{L}_1,\mathcal{L}_2)$
that maps to $\{\ast\}$ in Lemma \ref{lem: kernel cohomology}.(1),
we have constructed a twisted subgroup scheme 
of $\GG_{m,S}^2\subset\GL_{2,S}$.
The intersection of the $n$-torsion of this twisted form
with $\ker(\det)$ gives a twisted form of 
the standard embedding $\bmu_{n,S}\subset\SL_{2,S}\subset\GL_{2,S}$.
By Remark \ref{rem: independence of n}, this is the 
twisted form of the standard embedding $\bmu_{n,S}\subset\SL_{2,S}\subset\GL_{2,S}$
associated to $(\mathcal{L}_1,\mathcal{L}_2)$.
\end{Remark}

\subsection{\'Etale double covers}
\label{subsec: etale}
We represent a class $\eta\in\Hfl{1}(S,(\ZZ/2\ZZ)_S)$  by
a $(\ZZ/2\ZZ)_S$-torsor $f:S'\to S$, that is, by an \'etale
double cover.
As explained in Lemma \ref{lem: kernel cohomology}.(2),
we obtain an $\Norm_{\GL_{2,S}}(\GG^2_{m,S})$-torsor
and we obtain a twisted torus
$$
(\GG^2_{m,S})_\eta \,:=\,\GG_{m,S}\wedge S',
$$
which by definition is given by $(\GG^2_{m,S}\times_S S')/(\ZZ/2\ZZ)$,
where the action of $\ZZ/2\ZZ$ is the diagonal action that is conjugation
by the matrix \eqref{eq: splitting matrix} (which lies in the 
normalizer $\Norm_{\GL_{2,S}}(\GG_{m,S}^2)$)
on the first factor and the torsor action on the second factor.
Next, we note
that the induced action of conjugation respects $\det$
and is trivial on $\det(\GG_{m,S}^2)=\GG_{m,S}$.
On the kernel of $\det$, we obtain an induced twist, which
is given by the quadratic twist $\GG_{m,S}\wedge S'$
of $\GG_{m,S}$ with respect to inversion, see also 
Proposition \ref{prop: normalizers}.(3).
We thus obtain an exact sequence
\begin{equation}
\label{eq: twisted torus}
1\,\to\,\GG_{m,S}\wedge S'\,\to\, (\GG_{m,S}^2)_\eta \,
\stackrel{\det}{\longrightarrow}\, \GG_m \,\to\,1
\end{equation}
of group schemes over $S$.

\begin{Remark}
\label{rem: Weil 1}
 It is not difficult to see that there exist isomorphisms 
 $$
\Res_{S'/S}\GG_{m,S'} \,\cong\, \GG^2_{m,S}\wedge S' \,=\, (\GG_{m,S}^2)_\eta
 $$
 of group schemes over $S$, where the term on the left is
 the \emph{Weil restriction}, see, for example, \cite[Section 7.6]{BLR}.
 Then, the determinant $(\GG^2_{m,S})_\eta\to \GG_{m,S}$ gets
 identified with the norm $\Res_{S'/S}(\GG_{m,S'})\to\GG_{m,S}$, 
 which identifies the kernel $\GG_{m,S}\wedge S'$ of $\det$
 with the norm one subtorus 
 inside the Weil restriction.
\end{Remark}

\begin{Remark}
\label{rem: cohomology fiber}
By the general mechanism of non-abelian cohomology,
the fiber of \eqref{eq: split exact sequence} over the class 
$\eta\in\Hfl{1}(S,(\ZZ/2\ZZ)_S)$ is isomorphic
to the quotient of $\Hfl{1}(S,(\GG^2_{S,m})_\eta)$
by an action of $\Hfl{0}(S,(\ZZ/2\ZZ)_S)$,
see, for example, \cite[Chapter I.\S5]{SerreGaloisCohomology} in the
context of Galois cohomology or \cite[Chapitre III.\S3]{Giraud} in general.
\end{Remark}

By Lemma \ref{lem: kernel cohomology}.(2), the class $\eta=[S']$ maps via $\sigma$ 
to the distinguished element $\{\ast\}$ of 
$\Hfl{1}(S,\GL_{2,S})$ if and only
if the $\OO_S$-module $f_\ast\OO_{S'}$ is free.
By Theorem \ref{thm: brian} this then defines a twisted subgroup 
scheme of the standard torus $\GG_{m,S}^2\subset\GL_{2,S}$.
This is explicitly given by the twists
$$
\GG_{m,S}^2\wedge S' \,\to\, \GL_{2,S}\wedge S'
$$
where the twisted quotient on the right is via the diagonal
$\ZZ/2\ZZ$-action of $S'\to S$ and conjugation by the matrix
\eqref{eq: splitting matrix} on $\GL_{2}$.
By the proof of Lemma \ref{lem: Zariski twisted forms of torus}.(1),
we have $\GL_{2,S}\wedge S'\cong\GL(f_\ast\OO_{S'})\cong\GL_{2,S}$,
where the last isomorphism uses that $f_\ast\OO_S'$ is a free
$\OO_S$-module of rank 2.

\begin{Remark}
 \label{rem: Weil 2}
 This argument shows that if $S'\to S$ is an \'etale double
 cover, then the Weil restriction embeds as a maximal torus
 $$ 
   \Res_{S'/S}(\GG_{m,S'})\,\to\, \GL(f_\ast\OO_{S'})
 $$
 and this is also true if $f_\ast\OO_{S'}$ not a free $\OO_S$-module.
\end{Remark}

To pass from twisted subgroup schemes of $\GG_{m,S}^2\subset\GL_{2,S}$
to twisted forms of the standard embedding 
$\bmu_{n,S}\subset\SL_{2,S}\subset\GL_{2,S}$ with $n\geq3$
we use Corollary \ref{cor: independence of n} and
Remark \ref{rem: independence of n}:
Associated to the class $\eta=[S']\in\Hfl{1}(S,(\ZZ/2\ZZ)_S)$
that maps to $\{\ast\}$ in Lemma \ref{lem: kernel cohomology}.(2),
we have constructed a twisted form 
$(\GG_{m,S}^2)_\eta\subset\GL_{2,S}$ of $\GG_{m,S}^2\subset\GL_{2,S}$.
The intersection of the $n$-torsion of this twisted form
with $\ker(\det)$ gives a twisted form of 
the standard embedding $\bmu_{n,S}\subset\SL_{2,S}\subset\GL_{2,S}$.
By Remark \ref{rem: independence of n}, this is the 
twisted form of the standard embedding $\bmu_{n,S}\subset\SL_{2,S}\subset\GL_{2,S}$
associated to $\eta$.
Note that by \eqref{eq: twisted torus}, it is the $n$-torsion
of the kernel $\GG_{m,S}\wedge S'$.

\subsection{Summary}
\label{subsec: short summary}
Let us briefly summarize and generalize the previous discussion:
If $S$ is a 
scheme and $n\geq3$ is an integer, 
then subgroup schemes $G\subset\SL_{2,S}$ that are fppf locally
conjugate to the standard embedding $\bmu_{n,S}\subset\SL_{2,S}$
are classified up to $\GL_{2}(S)$-conjugacy by 
$$
\ker\left(\,
\Hfl{1}(S,\Norm_{\GL_{2,S}}(\bmu_n))\,\to\,\Hfl{1}(S,\GL_{2,S})
\,\right).
$$
Under the split surjective map
$$
\Hfl{1}(S,\Norm_{\GL_{2,S}}(\bmu_n)) \,\to\, \Hfl{1}(S,\ZZ/2\ZZ)
$$
the class $[G]$ of such a form maps 
to the class of some $\ZZ/2\ZZ$-torsor $f:S'\to S$, that is,
an \'etale double cover.
Using the splitting, there is an associated
subgroup scheme $G'\subset\SL_{2,S}$ that is locally in 
the fppf topology (in fact: \'etale topology) conjugate
to the standard embedding $\bmu_{n,S}\subset\SL_{2,S}$.
A construction of the form $G'=\bmu_{n,S}\wedge{S'}$ 
is described in detail in Section \ref{subsec: etale}.

After base changing to $S'$, the three classes of the standard embedding
$\bmu_{n,S'}\subset\SL_{2,S'}$, of $G_{S'}\subset\SL_{2,S'}$ and of
$G'_{S'}\subset\SL_{2,S'}$ in the middle of 
$$
1\,\to\,
\Hfl{1}(S',\GG_{m,S'}^2) \,\to\, \Hfl{1}(S',\Norm_{\GL_{2,S'}}(\bmu_n)) 
\,\stackrel{\pi}{\longrightarrow}\,\Hfl{1}(S',\ZZ/2\ZZ)\,\to\,1
$$
map to $\{\ast\}$ and thus, lie in the fiber $\pi^{-1}(\{\ast\})$.
Thus, $G_{S'}$ and $G'_{S'}$ are locally in the fppf topology
(in fact: Zariski topology) conjugate 
to the standard embedding $\bmu_{n,S'}\subset\SL_{2,S'}$.
A construction of the forms $G_{S'}$ and $G'_{S'}$
is described in detail in Section \ref{subsec: Zariski}.

\begin{Remark}
 \label{rem: topology}
 In view of Proposition \ref{prop: normalizers}.(2), we have
 isomorphisms of pointed cohomology sets
 $$
 \begin{array}{lcl}
\Het{1}(S,\Norm_{\GL_{2,S}}(\bmu_{n,S})) &\cong&\Hfl{1}(S,\Norm_{\GL_{2,S}}(\bmu_{n,S})) \\
\Het{1}(S,\Norm_{\GL_{2,S}}(\GG_{m,S}^2)) &\cong&\Hfl{1}(S,\Norm_{\GL_{2,S}}(\GG_{m,S}^2)),
\end{array}
 $$
 see \cite[Theorem III.3.9]{Milne}. 
 In particular, a subgroup scheme $G\subset\SL_{2,S}$ that is fppf locally
conjugate to the standard embedding $\bmu_{n,S}\subset\SL_{2,S}$
is already \'etale locally conjugate to the standard embedding $\bmu_{n,S}\subset\SL_{2,S}$
and the discussion preceding this remark make this precise.
\end{Remark}

\subsection{Dedekind rings}
Lemma \ref{lem: kernel cohomology} and Lemma \ref{lem: Zariski twisted forms of torus}
beg for the characterization of
$r$-tuples of invertible $\OO_S$-modules 
$\vec{\mathcal{L}}:=(\mathcal{L}_1,...,\mathcal{L}_r)$
that satisfy $\mathcal{L}_1\oplus...\oplus\mathcal{L}_r\cong\OO_S^{\oplus r}$.
Clearly, we have a map
$$
\omega:\Pic(S)^{\oplus r}\,\to\, \Pic(S), \quad 
\vec{\mathcal{L}}\,\mapsto\,\mathcal{L}_1\otimes...\otimes\mathcal{L}_r.
$$
Now, if $\mathcal{L}_1\oplus...\oplus{\mathcal{L}}_r\cong\OO_S^{\oplus r}$,
then taking determinants shows that $\vec{\mathcal{L}}$ must lie in the
kernel of $\omega$, that is, $\omega(\vec{\mathcal{L}})\cong\OO_S$.
The converse holds for spectra of Dedekind rings:

\begin{Theorem}[Steinitz]
\label{thm: Steinitz}
Let $R$ be a Dedekind domain and set $S:=\Spec R$.
\begin{enumerate}
\item If $\mathcal{E}$ is a locally free $\OO_S$-module of rank $r\geq1$,
then there exist invertible $\OO_S$-modules
$\mathcal{L}_1,...,\mathcal{L}_r$, such that
$$
 \mathcal{E}\quad\cong\quad\mathcal{L}_1\oplus...\oplus\mathcal{L}_r.
$$
\item For invertible $\OO_S$-modules 
$\vec{\mathcal{L}}:=(\mathcal{L}_1,...,\mathcal{L}_r)$
and $\vec{\mathcal{M}}=(\mathcal{M}_1,...\mathcal{M}_r)$, we have
$$
\bigoplus_{i=1}^r\,\mathcal{L}_i\,\cong\,\bigoplus_{i=1}^r\,\mathcal{M}_i
\quad\Leftrightarrow\quad
\omega(\vec{\mathcal{L}})\,\cong\,\omega(\vec{\mathcal{M}}).
$$
In particular, the direct sum on the left is a free $\OO_S$-module of rank $r$ 
if and only if $\omega(\vec{\mathcal{L}})\cong\OO_S$.
\end{enumerate}
\end{Theorem}

\begin{proof}
This is just a reformulation of a theorem of Steinitz \cite{Steinitz},
see \cite[\S22]{CR} or \cite[Chapter II.4]{FT} for modern accounts.
\end{proof}

We leave the following fancy reformulation to the
reader.

\begin{Corollary}
\label{cor: Steinitz}
    If $R$ is a Dedekind domain, set $S:=\Spec R$, 
    let $\GG_{m,S}=\GL_{1,S}\subset\GL_{r,S}$ be the subtorus 
    defined by
    \begin{equation}
    \label{eq: subtorus}
       t\,\mapsto\,\left(\begin{array}{cccc}
       t&0 & ... & 0 \\
       0&1 & .... & 0\\
       \multicolumn{4}{c}{\cdots} \\
       0 & 0 &... & 1 \end{array}\right) \,.
    \end{equation}
    The map \eqref{eq: GL2 cohomology} can be factored as
    $$
     \Hfl{1}(S,\GG_{m,S}^r)\quad\stackrel{\omega_\ast}{\longrightarrow}\quad
     \Hfl{1}(S,\GG_{m,S})\quad\to\quad \Hfl{1}(S,\GL_{r,S}),
    $$
    where the first map is surjective and the second map 
    is a bijection of pointed sets induced by the 
    inclusion \eqref{eq: subtorus}.\qed
\end{Corollary}

\begin{Example}
 This theorem is not true for non-affine Dedekind schemes,
 even if $r=2$ as the following example shows:
 Let $S=\PP^1_k$, where $k$ is a field.
 Then, $\mathcal{E}_m:=\OO_S(m)\oplus\OO_S(-m)$
 has trivial determinant for all $m\in\ZZ$
 and thus, the tuple $\vec{\mathcal{L}}_m:=(\OO_S(m),\OO_S(-m))$ lies in $\ker\omega$.
 On the other hand, we have $\dim_k H^0(S,\mathcal{E}_m)=m+1$ 
 if $m\geq1$ and $\dim_k H^0(S,\OO_S^{\oplus2})=2$,
 which shows that $\mathcal{E}_m\not\cong\OO_S^{\oplus 2}$ if $m\geq2$.
\end{Example}

\section{Examples of non-standard embeddings of $\mu_n$}
\label{sec: examples}

In Section \ref{sec: twisted forms of mun}, we described subgroup schemes
$G\subset\SL_{2,S}$ that are fppf locally conjugate to the standard 
embedding $\bmu_{n,S}\subset\SL_{2,S}\subset\GL_{2,S}$.
\begin{enumerate}
\item If $S$ is the spectrum of a field, then 
only the \'etale forms described in Section \ref{subsec: etale}
can show up and we described them in \cite{LS2}.
For non-trivial \'etale forms, the group scheme $G$ is not isomorphic to 
$\bmu_{n,S}$, but a quadratic twist.
\item Now, if $\Pic(S)$ is non-trivial, then also Zariski local forms as
described in Section \ref{subsec: Zariski} show up.
These have the property that $G$ is abstractly isomorphic to $\bmu_{n,S}$,
but that its embedding into $\SL_{2,S}\subset\GL_{2,S}$ 
is not $\GL_2(S)$-conjugate to the standard embedding; rather it is 
Zariski locally conjugate to the standard embedding.
We will refer to these as \emph{non-standard embeddings} of 
$\bmu_n\subset\SL_{2,S}\subset\GL_{2,S}$.
\end{enumerate}
The objective of this section is to give examples of such 
non-standard embeddings in the special case
where $S=\Spec\OO_K$, where $\OO_K$ is the ring
of integers in a number field $K$.

\subsection{Zariski local forms revisited}
Let $K$ be a number field, that is, a finite field extension of $\QQ$. 
Let $\OO_K$ be the ring of integers in $K$ and let $\Cl_K:=\Cl(\OO_K)\cong\Pic(\OO_K)$
be the \emph{class group} of $K$.
It is well-known that $\Cl_K$ is a finite group and thus, its order $h_K$, which is
called the \emph{class number} of $K$, is finite.
We refer to \cite[Kapitel I]{Neukirch} for details and proofs.

The following proposition can also be deduced from the general results
in Section \ref{sec: twisted forms of mun}, but its proof - which is independent
of the machinery introduced in the previous sections - leads the way
to find non-standard embeddings of $\bmu_n$ into $\SL_2$ in the case
of rings of integers in number fields.

\begin{Proposition}\label{prop:hK1-stnd-embeddings}
Let $K$ be a number field with $h_K=1$.
Let $G\subset \SL_{2,\OO_K}$ be a subgroup scheme that is $\GL_2(K)$-conjugate
to the standard embedding $\bmu_{n,\OO_K}\subset\SL_{2,\OO_K}\subset\GL_{2,\OO_K}$.
Then, $G\subset\SL_{2,\OO_K}$ is also $\GL_2(\OO_K)$-conjugate to 
$\bmu_{n,\OO_K}\subset\SL_{2,\OO_K}$   
\end{Proposition}

\begin{proof}
By assumption, there exists $A \in \GL_2(K)$, such that $G$ is of the form
$A \diag(u,u^{n-1}) A^{-1}$ and $u\in\bmu_n$
We may always scale to assume that $A$ has entries in $\calO_K$,
but then, its inverse may not have entries in $\OO_K$, that is, we may not
assume that $A$ lies in $\GL_2(\calO_K)$.
Setting 
$$
 A \,:=\,\left(\begin{array}{cc}
         a&b\\ c&d
\end{array}\right)\,\in\,\mathrm{Mat}_{2\times2}(\OO_K),
$$
we see 
$$
A \diag(u,u^{n-1}) A^{-1}\,=\,\frac{1}{\Delta}\left(\begin{array}{cc}
         uad-u^{n-1}bc& (u^{n-1}-u)ab \\ (u-u^{n-1})cd& u^{n-1}ad-ubc
\end{array}\right)
$$
where $\Delta := \det(A) = ad-bc$. 
All of these quantities lie in $\calO_K$. 
If $\Delta \in \calO_K^\times$ then we are done because $A$ lies in $\GL_2(\OO_K)$.
If $\Delta$ is not a unit, 
then there exists a prime ideal $\mathfrak{p}\subset\OO_K$ containing $\Delta$. 
As a result, $\idealp$ contains $ad$ and $bc$, so $A$ must have either a row or a 
column contained in $\idealp$. 
Since $h_K=1$, we have that $\idealp$ is a principal ideal, say,
generated by some irreducible element $p\in\OO_K$. 
Scaling down the appropriate row or column of $A$ by $p$, 
we have decreased $|\Norm(\det(A))|$, so by induction we are done.
\end{proof}

\begin{Remark}
The proof shows that even if we do not assume $h_K=1$, then it is still true that 
all embeddings of $\bmu_n$ in $\SL_{2,\calO_K}$ that are $\GL_2(K)$-conjugate to 
the standard embedding $\rho:\bmu_n\to\SL_2$ must be of the form 
$A \rho(\bmu_{n,K})  A^{-1}$ for some $A\in\mathrm{Mat}_{2\times2}(\OO_K)\cap\GL_2(K)$ 
that has a row or a column with entries 
in some prime ideal $\idealp\subset\OO_K$.
\end{Remark}

\subsection{Examples}
Next, we show the necessity of $h_K=1$ in 
Proposition \ref{prop:hK1-stnd-embeddings} by producing an example 
of an embedding $\bmu_{n,\calO_K}\subset\SL_{2,\calO_K}$ that is 
Zariski locally conjugate to the standard embedding but 
not globally conjugate to the standard embedding.

\begin{Example}\label{ex:non-stnd-mun}
Let $K=\QQ(\sqrt{10})$, whose ring of integers is $\calO_K=\ZZ[\sqrt{10}]$. 
Let $\alpha:=4+\sqrt{10}$ with Galois conjugate 
$\overline{\alpha}=4-\sqrt{10}$
and consider the matrix
$$
A\,:=\,\left(\begin{array}{cc}

        3&\alpha\\
        \overline{\alpha} & 3
    \end{array}\right)
    \,\in\,\mathrm{Mat}_{2\times2}(\OO_K)\cap\GL_{2}(K).
$$
Note that the principal ideal $(3)\subset\calO_K$ factors into the product of 
non-principal prime ideals $(3,\alpha)\cdot (3,\overline{\alpha})$. 
In particular, every row or column of $A$ has entries living in a prime ideal. 
We see 
$$
\det A \,=\,9-\Norm(\alpha)\,=\,3,
$$
where $\Norm$ denotes the norm,
which divides $3\alpha$, $3\overline{\alpha}$, $3^2$, and $\alpha\overline{\alpha}=6$. 
Hence,
$$
A\,\left(\begin{array}{cc}
        u&0\\
        0&u^{-1}
    \end{array}\right)\, A^{-1}
$$
has entries in $\calO_K$. 
From this, we obtain an embedding of $\bmu_{n,\OO_K}$ into $\SL_{2,\OO_K}$.

\begin{enumerate}
\item We claim that this yields an embedding $\bmu_{n,\calO_K}\to\SL_{2,\calO_K}$ that 
is not $\GL_{2,\calO_K}$-conjugate to the standard embedding 
$\bmu_{n,\calO_K}\subset\SL_{2,\OO_K}$. 

If it were, then there would be a matrix $B\in \GL_2(\calO_K)$ such that
$$
    B^{-1}A\left(\begin{array}{cc}
        u&0\\
        0&u^{-1}
    \end{array}\right)A^{-1}B\quad=\quad 
    \left(\begin{array}{cc}
        u&0\\
        0&u^{-1}
    \end{array}\right)
$$
Thus, $B^{-1}A$ is in $\GL_2(K)$ and it centralizes $\bmu_{n,K}$, hence is a diagonal matrix. 
Thus, we see
$$
    B\quad=\quad A\cdot \diag(\lambda_1,\lambda_2)
$$
with $\lambda_i\in K^\times$. 
Since $B$ has entries in $\calO_K$, we find 
$3\lambda_1\in\OO_K$ and $\overline{\alpha}\lambda_1\in\calO_K$. 
As a result, we find $9\Norm(\lambda_1)\in\ZZ$ and $6\Norm(\lambda_1)\in\ZZ$ 
and hence, $3\Norm(\lambda_1)\in\ZZ$. 
Similarly, we find $3\Norm(\lambda_2)\in\ZZ$. 
Next, 
$$
    \pm1\,=\,\Norm(\det B)\,=\,(3\Norm\lambda_1)\cdot (3\Norm\lambda_2)
$$
is a product of integers and thus,
$$
    3\Norm(\lambda_i)\,=\,\pm1.
$$
Since we also know $3\lambda_1=a+b\sqrt{10}\in\calO_K$, we see
$$
    a^2-10b^2\,=\,\Norm(3\lambda_1)\,=\,\pm3
$$
with $a,b\in\ZZ$. 
Reducing modulo $5$, we obtain a contradiction.

\item Thus, we have produced a non-standard embedding of $\bmu_{n,\calO_K}$ into $\SL_{2,\calO_K}$. 
We note that it is Zariski locally conjugate to the standard embedding. 
Indeed, we have
$$
    2\cdot 3 \,=\, \alpha\overline{\alpha},
$$
so upon inverting $\alpha$, we see $\overline{\alpha}$ is a multiple of $3$. 
Thus, we may take $\lambda_1=1/3$ and $\lambda_2=1$. 
Similarly, if we invert $\overline{\alpha}$, we may take $\lambda_1=1$ and $\lambda_2=1/3$. 
Lastly, if we invert $3$, then since $\det A=3$ we may take $\lambda_1=\lambda_2=1$. 
These define three Zariski open patches that cover $\Spec\calO_K$, 
as $(3,\alpha,\overline{\alpha})=\calO_K$ and over each patch,
our non-standard embedding is conjugate to the standard embedding.
\end{enumerate}
\end{Example}

\begin{Example}\label{ex:more-example-non-stnd-mun-invars}
  Let $q>3$ be a prime with $q\equiv3\pmod{4}$ such that the integer 
  $d:=q-1$ is square-free. 
  Let $K=\QQ(\sqrt{-dq})$. 
  We set
  $$
    A \,:=\,\left(\begin{array}{cc}
        q & \sqrt{-dq}\\
        -\sqrt{-dq} & q
    \end{array}\right).
  $$
  The columns generate a non-principal ideal in $\calO_K$ by
  \cite[Corollary 2.7]{KConrad}. 
  Conjugating the standard embedding $\rho:\bmu_n\to\SL_{2,K}$, $u\mapsto \diag(u,u^{n-1})$ by $A$, 
  we see that
    $$
A\,\left(\begin{array}{cc}
        u&0\\
        0&u^{n-1}
    \end{array}\right)\, A^{-1}
$$
has entries in $\calO_K$ and it is invertible in $K$. 
From this, we obtain a non-standard embedding of $\bmu_n$ into $\SL_{2,\OO_K}$.
We want to compute the associated invariant subring
$$
 \OO_K[x,y]^{\bmu_n}
$$
and note that if we set
$$
v'\,:=\,Ax\,=\, qx + \sqrt{-dq}\cdot y,\quad
w'\,:=\,Ay\,=\,-\sqrt{-dq}\cdot x+qy,
$$
then 
\begin{equation}\label{eqn:O_Kq-1}
    \OO_K[q^{-1}][x,y]^{\bmu_n}=\OO_K[q^{-1}][v'^n,w'^n,v'w'].
\end{equation}
We have $\det A=q(q-d)=q$, which is not a unit in $\OO_K$.
Now, these 3 invariants generate the invariant ring over
the principal open subset $D(\det A)=\Spec\OO_K[q^{-1}]$ 
of $\Spec\OO_K$.
In particular, 
$$
\Spec \OO_K[q^{-1}][x,y]^{\bmu_n} \,\to\, \Spec \OO_K[q^{-1}]
$$
is a family of RDPs and the fiber over each point is an RDP of type $A_{n-1}$.

To see what happens outside this principal open set, we set
$L=K(\sqrt{-q})$, which is a Galois extension $L/K$ with
group $\Gal(L/K)\cong\ZZ/2\ZZ$ and we let $\sigma\in\Gal(L/K)$ 
be the generator.
  We note that
    $$
-\frac{1}{\sqrt{-q}}A \,=\,\left(\begin{array}{cc}
        \sqrt{-q}&-\sqrt{d}\\
        \sqrt{d}&\sqrt{-q}
    \end{array}\right)
$$
    lies in $\GL_2(\calO_L)$. 
    The invariant ring over $\OO_L$ can thus be computed as follows
    (see also the proof of Lemma \ref{lem: invariant well-defined}):
    $$
    \calO_L[x,y]^{\bmu_n}\,=\,\calO_L[v^n,vw,w^n],
    $$
 where
    $$
    v\,:=\frac{-1}{\sqrt{-q}}v'\,=\,\,\sqrt{-q}x - \sqrt{d}y,\quad 
    w\,:=\,\frac{-1}{\sqrt{-q}}w'\,=\,\sqrt{d}x + \sqrt{-q}y.
    $$
    The generator $\sigma\in\Gal(L/K)$ sends $\sqrt{-q}$ and $\sqrt{d}$,
    as well as $v$ and $w$ to their negatives.
\begin{enumerate}
\item[(a)] If $n$ is even, then we find
    $$
    \calO_K[x,y]^{\bmu_n} \,=\, \calO_K[v^n,vw,w^n].
    $$
    If we think of $\Spec\OO_K[x,y]^{\bmu_n}\to\Spec\OO_K$ as a family
    of RDPs, then the fiber over each point is an RDP of type $A_{n-1}$.
\item[(b)] If $n$ is odd, then we have more interesting behavior. 
    By \cite[Theorem 3.1]{KConrad}, we have an equality of $\calO_K$-modules
    $$
    \calO_L \,=\,\calO_Ke_1\,\oplus\,\mathfrak{q}e_2,
    $$ 
    where $\mathfrak{q}=q\calO_K+\sqrt{-dq}\calO_K$ 
    and where
    $$
    e_1\,=\,\frac{1+\sqrt{-q}}{2}\quad\textrm{and}\quad e_2\,=\,\frac{1}{\sqrt{-q}}.
    $$
    The generator $\sigma\in\Gal(L/K)$ acts by 
    $$
    \sigma(e_1)\,=\,e_1+qe_2,\quad\sigma(e_2)\,=\,-e_2.
    $$
    It follows that $\alpha\in\calO_L$ has $\sigma(\alpha)=-\alpha$ if and only if $\alpha\in\mathfrak{q}e_2$. 
    Hence,
\begin{eqnarray*}
        \calO_K[x,y]^{\bmu_n} &=&\calO_L[v^n,vw,w^n]^{\Gal(L/K)}\\
        &=&\calO_K[vw,  \sqrt{d} v^n, \sqrt{d} w^n, \sqrt{-q} v^n, \sqrt{-q} w^n].
\end{eqnarray*}
Let us justify why equality holds. 
We have 
$$
\calO_L[v^n,vw,w^n]^{\Gal(L/K)}\subset (L[x,y]^{\mu_n})^{\Gal(L/K)}\cap \calO_L[x,y]^{\Gal(L/K)}.
$$ 
Then $\calO_L[x,y]^{\Gal(L/K)}=\calO_K[x,y]$ and 
$$
(L[x,y]^{\mu_n})^{\Gal(L/K)}=K[x,y]^{\mu_n}
$$ 
by linear reductivity of $\Gal(L/K)$ over $K$, so 
\begin{equation}\label{eqn:Gal-invariants}
    \calO_L[v^n,vw,w^n]^{\Gal(L/K)}\subset \calO_K[x,y]^{\bmu_n}.
\end{equation}
In fact, since $\Gal(L/K)=\ZZ/2\ZZ$ is linearly reductive over $\calO_K[2^{-1}]$, 
we see equation \eqref{eqn:Gal-invariants} becomes an equality after inverting $2$. 
Furthermore, we explicitly computed $\calO_K[x,y]^{\bmu_n}$ after inverting $q$, 
given in equation \eqref{eqn:O_Kq-1}. 
Since $q\calO_K+2\calO_K=\calO_K$, we therefore see equation \eqref{eqn:Gal-invariants} 
is an equality.

    Next, letting
    $$
    A\,:=\,vw,\quad B\,:=\,\sqrt{d} v^n,\quad 
    C\,:=\,\sqrt{d} w^n, \quad B'\,:=\,\sqrt{-q} v^n, \quad C'\,:\,=\sqrt{-q} w^n,
    $$
    we see
    $$
        \calO_K[x,y]^{\bmu_n}=\calO_K[A,B,C,B',C']/J,
    $$
    where $J$ is the ideal generated by
    $$
    \sqrt{-dq}B'+qB,\quad\sqrt{-dq}C'+qC,\quad \sqrt{-dq}B-dB',\quad \sqrt{-dq}C-dC',
    $$
    $$
    dA^n-BC,\quad qA^n+B'C',\quad\sqrt{-dq}A^n-B'C,\quad \sqrt{-dq}A^n-BC'.
    $$
    Again, we think of 
    $$
    \Spec\OO_K[x,y]^{\bmu_n} \,\to\, \Spec\OO_K
    $$
    as a family of RDPs.
    It follows from Proposition \ref{prop: fibers} that the fibers over $\calO_K$
    are RDPs of type $A_{n-1}$, but this can also be seen
    by explicit computation:
    \begin{enumerate}
    \item Over $K$,   
    we have that $B, B'$ and $C,C'$ are multiples of each
    other. 
    Thus, we can eliminate two of them and obtain
    \begin{eqnarray*}
    K[x,y]^{\bmu_n} &\cong&K[A,B,C]/(dA^n-BC)\\
    &\cong& K[A,B',C']/(qA^n+B'C'),
    \end{eqnarray*}
    which is an RDP of type $A_{n-1}$ over $K$.
    (Note that in the first equation, one can divide by $d$ and after 
    rescaling $\widetilde{B}=B/d$, one obtains the usual equation $A^n-\widetilde{B}C$
    and similarly, for the second equation.)
    \item Similarly, if $\idealp\subset\OO_K$ is a prime ideal
    not lying over $(q)\subset\ZZ$,
    then we obtain 
    $$
    (\OO_K/\idealp)[x,y]^{\bmu_n} \,\cong\,
    (\OO_K/\idealp)[A,B,C]/(dA^n-BC)
$$
    which is an RDP of type $A_{n-1}$ 
    over $\OO_K/\idealp$.
    \item Finally, if $\idealp\subset\OO_K$ is a prime ideal
    not lying over a prime dividing $d\in\ZZ$,
    then we obtain 
    $$
    (\OO_K/\idealp)[x,y]^{\bmu_n} \,\cong\,
    (\OO_K/\idealp)[A,B',C']/(qA^n+B'C')
$$
    which is an RDP of type $A_{n-1}$ over
    $\OO_K/\idealp$.
\end{enumerate}
\end{enumerate}
\end{Example}

\begin{Remark}
\label{rem: noether}
If $k$ is field of characteristic zero and if $G$ is a finite group $G$ acting linearly
$k[x_1,...,x_n]$, then a classical theorem of Noether \cite{Noether}
states that the invariant subring 
$k[x_1,...,x_n]^G\subseteq k[x_1,...,x_n]$
can be generated as a $k$-algebra by finitely many homogenous elements of degree 
at most the order of $G$.
Noether's degree bound still holds for linear actions of finite and linearly reductive group
schemes $G\to\Spec k$ on polynomial rings $k[x_1,...,x_n]$
over arbitrary ground fields $k$ if `order' is replaced by `length' of the group scheme
\cite{KLO}.
In the previous example, we considered $\bmu_n$-actions on $\OO_K[x,y]$ and
found 5 elements of degree $\leq n$ that generate $\OO_K[x,y]^{\bmu_n}$ as an
$\OO_K$-algebra.
Thus, Noether's degree bound also holds in these examples.
However, it seems that we may need more generators and relations than
over fields.
\end{Remark}

\section{Rings of integers in number fields}
\label{sec: number fields}

In this section, we let $K$ be a number field with ring of integers $\OO_K$.
We want to describe the set
$$
\Klein(n,\OO_K)
$$
of finite flat linearly reductive subgroup schemes of $\SL_{2,\OO_K}$ of length $n$
up to $\GL_{2}(\OO_K)$-conjugacy.
By Theorem \ref{thm: key} (note that $\OO_K$ is an excellent Dedekind domain
and that it has a prime ideal with residue characteristic $2$), 
the set $\Klein(n,\OO_K)$ consists of
fppf locally trivial conjugates of the standard embedding
$\bmu_{n,\OO_K}\subset\SL_{2,\OO_K}\subset\GL_{2,\OO_K}$.
We gave a cohomological description of these forms
 in Corollary \ref{cor: key}.

\subsection{Finiteness}
\label{subsec: finiteness}
We start with $n=2$, which is easy and exceptional.

\begin{Proposition}
\label{prop: Klein(2,O)} 
 Let $K$ be a number field with ring of integers $\OO_K$.
 Then, $\Klein(2,\OO_K)$ is a singleton, which can be represented
 by the standard embedding $\bmu_{2,\OO_K}\subset\SL_{2,\OO_K}$.
\end{Proposition}

\begin{proof}
Let $G\subset\SL_{2,\OO_K}$ be a finite, flat, and 
linearly reductive closed subgroup scheme of length $2$ over $\Spec\OO_K$.
By Theorem \ref{thm: key}, $G$ is a twisted subgroup scheme of 
the standard embedding $\bmu_{2,\OO_K}\subset\SL_{2,\OO_K}$.
Since Proposition \ref{prop: normalizers} gives
$\Norm_X(\bmu_2)=X$ for $X\in\{\SL_2,\GL_2\}$,
it follows from Corollary \ref{cor: key} that 
there are no non-trivial twisted subgroup schemes of the standard embedding.
\end{proof}

\begin{Remark}
 The proof actually shows that also 
 the set of finite flat linearly reductive subgroup
 schemes of $\SL_{2,\OO_K}$ up to $\SL_{2,\OO_K}$-conjugacy
 consists of only one element.
 Moreover, this proposition can be generalized to twisted subgroup
 schemes of $\bmu_{n,\OO_K}\subset\GL_{2,\OO_K}$ for all $n\geq2$, 
 where the embedding is given by $\psi\oplus\psi$ 
 (notation as in Section \ref{sec: key}):
 there are no non-trivial twisted subgroup schemes.
 If $n\geq3$, then these subgroup schemes of $\GL_{2,\OO_K}$
 do not lie inside $\SL_{2,\OO_K}$.
\end{Remark}

We will now assume $n\geq3$.
By Theorem \ref{thm: brian} and Corollary \ref{cor: key}, 
we want to understand
\begin{equation}
\label{eq: kernel again}
\ker\left(
\Hfl{1}(\OO_K,\Norm_{\GL_{2,\OO_K}}(\bmu_{n,\OO_K}))\,\to\,
\Hfl{1}(\OO_K,\GL_{2,\OO_K})
\right)
\end{equation}
and by Proposition \ref{prop: normalizers}, the term on the left
sits in a split short exact sequence
\begin{equation}
\label{eq: sequence again}
1\,\to\,\Hfl{1}(\OO_K,\GG_{m,S}^2)\,\to\,\Hfl{1}(S,\Norm_{\GL_{2,S}}(\bmu_{n,\OO_K}))
\,\to\,\Hfl{1}(\OO_K,\ZZ/2\ZZ)\,\to\,1.
\end{equation}

The term on the left can be identified with 
$\Pic(\OO_K)^{\oplus2}\cong\Cl(\OO_K)^{\oplus2}$,
where $\Cl_K:=\Cl(\OO_K)$ is the class group of $K$, whose
order $h_K$ is the class number of $K$.
The term on the right of \eqref{eq: sequence again}
classifies \'etale double covers $S'\to\Spec\OO_K$,
which is the same as finite flat 
unramified extensions $\OO_K\to R$ of degree $2$.

\begin{Caution}
\label{caution cft}
 Let $K\subseteq L$ be an extension of number fields
 and $\OO_K\subseteq\OO_L$ be their rings of integers.
 We will say that $L/K$ is \emph{unramified} if 
 $\OO_K\subseteq\OO_L$ is unramified in the sense of
 commutative algebra.
 This is more general than how the term is sometimes used in
 global class field theory, where one also requires 
 all real places
 of $K$ to stay real in $L$.
\end{Caution}

We briefly digress on class field theory and refer to
\cite[Kapitel VI.6]{Neukirch} for details and proofs:
If $K$ is a number field, then there is the
\emph{Hilbert class field} $H$, which 
is the maximal abelian unramified extension of $K$, all
of whose real places stay real, 
and there exists the \emph{ray class field of modulus 1}
$$
 K\quad \subseteq\quad H\quad\subseteq\quad K^1,
$$
which is the maximal abelian unramified extension of $K$ 
with no restrictions on the infinite places.
Both are finite Galois extensions of $K$ and
with Galois groups
$$
\Gal(H/K) \,\cong\,\Cl_K
\mbox{ \quad and \quad }
\Gal(K^1/K)\,\cong\,\Cl_K^1.
$$
Here, $\Cl_K$ is the class group and
$\Cl^1_K$ is the ray class group of modulus 1,
which sits in an exact sequence
\begin{equation}
\label{eq: ray class modulus 1}
0\,\to\,\OO_K^\times/(\OO_K)^\times_+\,\to\,
\prod_{\idealp\mbox{ real}} \RR^\times/\RR^\times_+
\,\to\,\Cl^1_K\,\to\,\Cl_K\,\to\,0.
\end{equation}
In particular, $\Cl^1_K$ is a finite abelian group
of order $2^{r-t}\cdot h_K$ and we have
$\Gal(K^1/H)\cong(\ZZ/2\ZZ)^{r-t}$,
where $r$ is the number of real places of $K$
and where $\OO_K^\times/(\OO_K)^\times_+$
is of size $2^t$.

\begin{Proposition}
\label{prop: unramified extensions of degree 2}
  Let $K$ be a number field with ring of integers $\OO_K$
  and let $\Cl_K^1$ be its ray class group of modulus 1.
  Then, there exists a bijection of pointed sets between
  \begin{enumerate}
      \item  $\Hfl{1}(\OO_K,\ZZ/2\ZZ)$.
      \item Unramified ring extension $\OO_K\to R$ of degree $2$.
      \item Homomorphisms $\Cl_K^1\to\ZZ/2\ZZ$.
  \end{enumerate}
  The distinguished elements in these sets are the trivial $\ZZ/2\ZZ$-torsor,
  the trivial unramified degree $2$ extension of $\OO_K$, 
  and the trivial homomorphism
  $\Cl^1_K\to\ZZ/2\ZZ$, respectively.
  
  In particular, all three sets are non-empty and finite.
  They consist of one element if and only if $\Cl^1_K$ is of 
  odd order.
\end{Proposition}

\begin{proof}
The bijection $(1)\leftrightarrow(2)$ is clear.
To give a non-trivial unramified extension $\OO_K\to R$ of degree 2
is equivalent to giving a quadratic field extension $K\subset L\subseteq K^1$
(note that $\Gal(K^1/K)$ is abelian so that we do not have to worry about
conjugate embeddings of $L$ into $K^1$),
which is equivalent to giving a surjective homomorphism $\Cl^1_K\to\ZZ/2\ZZ$.
This gives the bijection $(2)\leftrightarrow(3)$.

Finiteness of the three sets follows from the finiteness of $\Cl^1_K$.
In particular, these sets consists of one element only
if and only if $\Cl^1_K$ admits
no surjective homomorphism to $\ZZ/2\ZZ$, that is, if and only if
$\Cl^1_K$ is of odd order.
\end{proof}

\begin{Theorem}
\label{thm: finite}
  Let $K$ be a number field with ring of integers $\OO_K$ and let 
  $n\geq3$ be an integer.
  \begin{enumerate}
  \item The set $\Klein(n,\OO_{K})$ 
     consists of $\GL_2(\OO_{K})$-conjugacy classes of 
     subgroup schemes $G\subset\SL_{2,\OO_{K}}$
     that are locally in the fppf topology conjugate to
     the standard
     embedding $\bmu_{n,\OO_{K}}\subset\SL_{2,\OO_{K}}$. 
  \item The set $\Klein(n,\OO_K)$ is finite, satisfies the estimate
  $$
    |\Klein(n,\OO_K)|\quad\geq\quad h_K,
  $$
  and its cardinality is independent of $n$.
  \item The set $\Klein(n,\OO_K)$ consists of one element
  if and only if $\Cl_K^1$ is trivial.
  In this case, $\Klein(n,\OO_K)$ consists only of the 
  $\GL_{2,\OO_K}$-conjugacy class of the standard embedding
  $\bmu_{n,\OO_K}\subset\SL_{2,\OO_K}$.
  \end{enumerate}
\end{Theorem}

\begin{proof}
(1) is Corollary \ref{cor: key}.

(2) Independence of $n$ for $n\geq3$ follows from 
Corollary \ref{cor: key} and Corollary \ref{cor: independence of n}.
To show finiteness, we have to show that the kernel of
\eqref{eq: kernel again} is finite.
Clearly, it suffices to show that $\Hfl{1}(\OO_K,\Norm_{\GL_2}(\bmu_n))$
is finite.
To do so, we use the short exact sequence \eqref{eq: sequence again}.
The set on the right is non-empty and finite 
by Proposition \ref{prop: unramified extensions of degree 2}.
The fiber over the trivial class is $\Hfl{1}(\OO_K,\GG_m^2)\cong\Cl(\OO_K)^2$,
which is finite.
In fact, this fiber intersected with the kernel of the map
$$
 \Hfl{1}(\OO_K,\GG_m^2) \,\to\, \Hfl{1}(\OO_K,\Norm_{\GL_{2,\OO_K}}(\bmu_{n,\OO_K}))
 \,\to\, \Hfl{1}(\OO_K,\GL_{1,\OO_K})
$$ 
can be identified with the kernel of the map $\Cl_K^2\to\Cl_K$ given
by $(x,y)\mapsto x+y$, see, Corollary \ref{cor: Steinitz}.
In particular, this intersection is isomorphic to $\Cl_K$ and
and shows that the cardinality of $\Klein(n,\OO_K)$ is at least
equal to $h_K$.

If $\eta\in\Hfl{1}(\OO_K,\ZZ/2\ZZ)$ is a non-trivial class, then
the fiber over $\eta$ is the quotient of $\Hfl{1}(\OO_K,(\GG_m^2)_\eta)$ by an
action of $\Hfl{0}(\OO_K,(\ZZ/2\ZZ)_\eta)$, see
Remark \ref{rem: cohomology fiber}.
We represent the class $\eta$ 
by some unramified extension $\OO_K\subset\OO_L$
of degree 2 of rings of integers of number fields $K\subset L$.
We have
$$
\Hfl{1}(\OO_K,(\GG_{m,\OO_K}^2)_\eta)\,\cong\,
\Hfl{1}(\OO_K,\Res_{\OO_L/\OO_K}(\GG_{m,\OO_L})),
$$
where the first isomorphism follows from base change 
and where $\Res$ denotes the Weil restriction, see Remark
\ref{rem: Weil 1} and Remark \ref{rem: Weil 2}.
Since $\OO_K\subset\OO_L$ is finite and \'etale, we have
$$
\Hfl{1}(\OO_K,\Res_{\OO_L/\OO_K}(\GG_{m,\OO_L})) \,\cong\,
\Hfl{1}(\OO_L,\GG_{m,\OO_L})\,\cong\, \Cl_L,
$$
Since $\Cl_L$ is finite, so is the fiber over $\eta$.
Putting all this together, we conclude that $\Klein(n,\OO_K)$
is finite.

(3)
If $\Klein(n,\OO_K)$ consists of one element, then
$|\Klein(n,\OO_K)|\geq h_K$ in (2) implies
$h_K=1$, that is, $\Cl_K$ is trivial.
Theorem \ref{thm: Steinitz} then
implies that every locally free $\OO_K$-module of rank 2
is trivial, that is, $\Hfl{1}(\OO_K,\GL_{2,\OO_K})=\{\ast\}$.
Using Lemma \ref{lem: kernel cohomology}.(2), we see that
every element of $\Hfl{1}(S,\ZZ/2\ZZ)$ gives
rise to subgroup scheme of $\SL_{2,\OO_K}$ that is
fppf locally conjugate to the standard embedding.
Thus, $\Hfl{1}(\OO_K,\ZZ/2\ZZ)$ consists of one element and thus, 
$\Cl^1_K$ must be of odd order by 
Proposition \ref{prop: unramified extensions of degree 2}.
Thus, if $\Klein(n,\OO_K)$ consists of one element,
then $\Cl_K$ must be trivial and $\Cl^1_K$ must be of odd order.
Using \eqref{eq: ray class modulus 1}, this implies that $\Cl^1_K$ is trivial.

Conversely, if $\Cl_K^1$ is trivial, then 
$\Hfl{1}(\OO_K,\ZZ/2\ZZ)$ is trivial by 
Proposition \ref{prop: unramified extensions of degree 2}.
Also, if $\Cl_K^1$ is trivial, then so is $\Cl_K$, which implies
that the fiber over the trivial class 
in \eqref{eq: sequence again} is trivial,
which implies that $\Klein(n,\OO_K)$ has only one element.
\end{proof}

\begin{Remark}
For $n\geq3$ we have $|\Klein(n,\OO_K)|\geq h_K$, which implies that
that this cardinality, although finite, 
is not bounded if $K$ runs through all number fields.
In fact, already the class numbers of the cyclotomic fields $\QQ(\zeta_n)$
are unbounded. 

It would be interesting to determine $|\Klein(n,\OO_K)|$:
The proof shows that this might lead to a formula that involves
the number $r-t$ from \eqref{eq: ray class modulus 1}
and the  class numbers $h_L$ of every unramified extension $L/K$ of degree $\leq2$.
\end{Remark}

\begin{Example} 
By \eqref{eq: ray class modulus 1}, the
ray class group $\Cl_K^1$ of modulus 1 is trivial
if and only if $h_K=1$ and $r=t$.
This is the case for $K=\QQ$ or for 
number fields with $h_K=1$ and $r=0$, which include
the imaginary quadratic
fields with $h_K=1$.
\end{Example}

\subsection{Arithmetic RDP singularities}
The linearly reductive quotient singularity associated to the standard
embedding of $\bmu_{n,\OO_K}\subset\SL_{2,\OO_K}\subset\GL_{2,\OO_K}$ 
is the spectrum of
\begin{equation}
\label{eq: An-singularity}
\OO_K[x,y]^{\bmu_{n,\OO_K}} \,\cong\, \OO_K[a,b,c]/(c^n-ab),
\end{equation}
where $a=x^n$, $b=y^n$, and $c=xy$.
If $\idealp\in\Spec\OO_K$ is a prime ideal with residue field
$\kappa(\idealp)$, then we obtain $\kappa(\idealp)[a,b,c]/(c^n-ab)$
as fiber over $\idealp$,
whose spectrum is an RDP singularity of type $A_{n-1}$.
This is the cyclic quotient singularity of type $\frac{1}{n}(1,n-1)$ 
and it can also be described in terms of toric geometry.

Now, let $G\subset\SL_{2,\OO_K}$ be a finite flat linearly reductive
subgroup scheme of length $n\geq3$, that is, $G\in\Klein(n,\OO_K)$.
We have the associated invariant subring
$$
\OO_K[x,y]^G\,\subset\,\OO_K[x,y],
$$
whose relative spectrum  is a family of RDP singularities
over $\Spec\OO_K$.

Let $[G]$ be the class of $G$ inside \eqref{eq: kernel again},
which maps under \eqref{eq: sequence again} to 
a class in $\Hfl{1}(\OO_K,\ZZ/2\ZZ)$, which we can represent
by an unramified ring extension $\OO_K\subset R$ of degree 2.

If this unramified extension $\OO_K\subset R$ is trivial, that is, if
$R\cong\OO_K\times\OO_K$, then $[G]$ maps to $\{\ast\}\in\Hfl{1}(\OO_K,\ZZ/2\ZZ)$
and thus, as explained in Section \ref{subsec: short summary},
$G\subset\SL_{2,\OO_K}$ is Zariski locally conjugate to the standard
embedding $\bmu_{n,\OO_K}\subset\SL_{2,\OO_K}$.
In particular, $G$ is isomorphic to $\bmu_{n,\OO_K}$ as a group
scheme over $\Spec\OO_K$.
Let  $\idealp\subset\Spec\OO_K$ and let $\OO_{K,\idealp}$ the
associated localization.
Since the Picard group of the latter is trivial, it follows
from the discussion in Section \ref{subsec: Zariski}
that $G\subset\SL_{2,\OO_K}$ becomes conjugate to the standard
embedding $\bmu_{n,\OO_K}\subset\SL_{2,\OO_K}$ over $\Spec\OO_{K,\idealp}$.
Thus, the invariant ring of $\OO_K[x,y]^G$ becomes isomorphic 
to \eqref{eq: An-singularity} over $\Spec\OO_{K,\idealp}$
and in particular, we obtain an RDP singularity
of type $A_{n-1}$ over $\kappa(\idealp)$.

\begin{Proposition}
\label{prop: fibers}
  We keep assumptions and notations and assume that $\OO_K\subset R$ is 
  a non-trivial extension.
  In this case, $R$ is the ring of integers $\OO_L$ inside
  a number field $L$ that is a quadratic Galois extension of $K$.
  \begin{enumerate}
  \item Let $\idealp\subset\OO_K$ be a non-zero prime ideal.
  Then $\idealp$ is unramified in $\OO_L$ and there are two cases:
 \begin{enumerate}
 \item If $\idealp$ is inert in $\OO_L$, then
 $$
 G\times_{\Spec\OO_K}\Spec\kappa(\idealp) \,\cong\, \bmu_{n,\kappa(\idealp)}
 $$
 and the associated quotient singularity over $\kappa(\idealp)$
 is an RDP of type $A_{n-1}$.
 \item If $\idealp\in\Spec\OO_K$ is split in $\OO_L$, then 
  $$
 G\times_{\Spec\OO_K}\Spec\kappa(\idealp) \,\cong\, \bmu_{n,\kappa(\idealp)}\wedge\Spec \kappa',
 $$
 where $\kappa'$ denotes the unique quadratic extension field of $\kappa(\idealp)$
 and the twist is with respect to inversion on $\bmu_n$.

 The associated quotient singularity over $\kappa(\idealp)$
 is an RDP of type $B_{\beta(n)}$. 
 Here, $\beta(n)$ is equal to $n/2$ if $n$ is even and equal to 
 $(n-1)/2$ if $n$ is odd.
 \end{enumerate}
  \item We have 
 $$
 G\times_{\Spec\OO_K}K \,\cong\, \bmu_{n,K}\wedge\Spec L,
 $$
 where the twist is with respect to the inversion.
 The associated quotient singularity is an RDP of type
 $B_{\beta(n)}$ over $K$.
 \end{enumerate}
 \end{Proposition}

 \begin{Remark}
  RDP singularities of type $B_m$ can only exist over non-closed
  fields and they were described and classified by Lipman 
  \cite{Lipman}.
  In \cite{LS2}, we showed that some of them arise as quotient
  singularities by twisted forms of $\bmu_n$, in which case
  $m=\beta(n)$ holds true.
 \end{Remark}

\begin{proof}
(1) Using linear reductivity, it is easy to see that 
taking twisted products and taking invariant rings commutes with base change.

Let $\idealq\in\Spec\OO_L$
be a prime lying over $\idealp$ 
and let $\kappa:=\kappa(\idealp)\subseteq\kappa':=\kappa(\idealq)$
be the corresponding extension of residue fields.
Since $\OO_K\subset\OO_L$ is unramified, so is $\idealp$
inside $\OO_L$.

(a) If $\idealp\in\Spec\OO_K$ is inert,
then $G_\idealp:=G\times\Spec\kappa(\idealp)$ is the trivial quadratic twist of
$\bmu_{n,\kappa(\idealp)}$.
In particular, $G_\idealp\subset\SL_{2,\kappa}$ is Zariski locally conjugate to 
the standard embedding $\bmu_{n,\OO_\kappa}\subset\SL_{2,\OO_\kappa}$.
Since the Picard group of the field $\kappa$ is trivial, 
the two actions are conjugate over $\Spec\kappa$
and thus, the invariant subring coincides with \eqref{eq: An-singularity}
over $\Spec\kappa$ and  we obtain an RDP singularity of type $A_{n-1}$.

(b) If $\idealp$ is split in $\OO_L$,
then the residue field extension $\kappa\subset\kappa'$ is Galois with
group $\ZZ/2\ZZ$ and $G_\idealp$ is the quadratic twist of $\bmu_{n,\kappa}$
with respect to $\Spec\kappa'\to\Spec\kappa$ and
the inversion, see Corollary \ref{cor: quadratic twist}.
The associated quotient singularity is an RDP of type $B_{\beta(n)}$
by \cite[Theorem 1.1]{LS2}.

This proves (1) and we leave (2), which is analogous to case (b) of part (1), 
to the reader.
\end{proof}

If $\Sigma$ is a (possibly empty or infinite) subset of $\Spec\OO_K$, then its \emph{Dirichlet density} 
is defined to be
$$
\delta_{\mathrm{Dir}}(\Sigma) \,:=\, \lim_{s\to1^+}
\frac{\sum_{\idealp\in\Sigma}|\kappa(\idealp)|^{-s}}{\sum_{\idealp\in\Spec\OO_K}|\kappa(\idealp)|^{-s}},
$$
provided this limit exists. 
If this limit exists, then also the \emph{natural density}
$$
\delta_{\mathrm{nat}}(\Sigma) \,:=\, \lim_{n\to\infty}
\frac{|\{\idealp\in\Sigma\,:\,|\kappa(\idealp)|\leq n \}|}{|\{\idealp\in\Spec\OO_K\,:\,|\kappa(\idealp)|\leq n \}|}
$$
exists and coincides with the Dirichlet density, see, for example,
\cite[Kapitel VII]{Neukirch}.
If $L/K$ is a finite Galois extension and $\Sigma\subseteq\Spec\OO_K$ is the set of primes that are totally
split in $\OO_L$, then Chebotarev's density theorem states that 
$\delta_{\mathrm{Dir}}(\Sigma)=\delta_{\mathrm{nat}}(\Sigma)=1/[L:K]$, see, for example,
\cite[Kapitel VII, Theorem (13.4)]{Neukirch}.
In particular, if $L/K$ is an unramified quadratic extension,
then all non-zero prime ideals of $\OO_K$
can be divided into two disjoint sets: the totally split ones and the inert ones.
Both sets have density $1/2$ and in particular, they are infinite.
Combining this with Proposition \ref{prop: fibers} and the discussion preceding it, 
we obtain the following.

\begin{Proposition}
    Let $G\in\Klein(n,\OO_K)$ and let
    $$
    f\quad:\quad\Spec\OO_K[x,y]^G\quad\to\quad\Spec\OO_K
    $$
    be the associated family of RDP singularities.
    We consider the subsets 
    $$
     \Sigma_A \subseteq\Spec\OO_K \mbox{ \quad and \quad }
     \Sigma_B \subseteq\Spec\OO_K
    $$
    that are defined by the property that the fiber $f^{-1}(\idealp)$
    is an RDP singularity of type $A_{n-1}$ and $B_{\beta(n)}$,
    respectively.
    Then, $\Sigma_A$ and $\Sigma_B$ partition $\Spec\OO_K$ and precisely one of the following two cases happens:
    \begin{enumerate}
        \item $G\cong\bmu_{n,\OO_K}$ and then, $\Sigma_A=\Spec\OO_K$,
        and $\Sigma_B=\emptyset$.
        \item $G\not\cong\bmu_{n,\OO_K}$ and then, $\Sigma_A$ and $\Sigma_B$
        are infinite sets and both have density $1/2$.\qed
    \end{enumerate}
\end{Proposition}

\subsection{The infinite places}
Let $\idealp$ be an infinite place of the number field $K$, 
let $K_\idealp$ be the corresponding completion, 
and let $G\in\Klein(n,\OO_K)$ with $n\geq3$.

As seen above, we have an unramified extension $\OO_K\subset R$
of degree 2 associated to $G$.
If it is trivial, then $G\subset\SL_{2,\OO_K}$ is Zariski locally
conjugate to the standard embedding $\bmu_{n,\OO_K}\subset\SL_{2,\OO_K}$
and thus, $G_K\subset\SL_{2,K}$ is $\GL_2(K)$-conjugate to 
the standard embedding $\bmu_{n,K}\subset\SL_{2,K}$.
In this case, we find that $G_{K_\idealp}\subset\SL_{2,K_\idealp}$ is 
$\GL_2(K_{\idealp})$-conjugate to the standard embedding 
$\bmu_{n,K_{\idealp}}\subset\SL_{2,K_{\idealp}}$ 
and thus, the associated quotient singularity is an RDP of type $A_{n-1}$
over $K_{\idealp}$.

We may thus assume that $\OO_K\subset R$ is a non-trivial extension,
that is, we may assume that $R=\OO_L$ for some unramified quadratic
Galois extension $K\subset L$.
Note that if $\idealp$ is real place then may not stay real
in $L$, see Caution \ref{caution cft}.
We have three cases:
\begin{enumerate}
\item  $\idealp$ is complex, that is $K_\idealp\cong\CC$.
Then $G_\idealp\subset\SL_{2,\CC}$
is $\GL_2(\CC)$-conjugate to the standard embedding of 
$\bmu_{n,\CC}\subset\SL_{2,\CC}$ 
and the associated quotient singularity is an RDP 
singularity of type $A_{n-1}$ over $\CC$.
\item $\idealp$ is real, that is $K_\idealp\cong\RR$ and then,
there are two subcases:
\begin{enumerate}
\item It stays real in $L$. 
Then, $G_\idealp\subset\SL_{2,\RR}$
is $\GL_2(\RR)$-conjugate to the standard embedding 
$\bmu_{n,\CC}\subset\SL_{2,\RR}$ 
and the associated quotient singularity is an RDP 
singularity of type $A_{n-1}$ over $\RR$.
\item $\idealp$ becomes complex in $L$ and then,
$G_\idealp$ is the unique quadratic twist $\bmu_{n,\RR}\wedge\Spec\CC$
with respect to inversion,
which is isomorphic to the constant group scheme $\ZZ/n\ZZ$ 
over $\RR$.
Its embedding into $\SL_{2,\RR}$
is $\GL_2(\RR)$-conjugate to
     $$
     \begin{array}{ccc}
 \ZZ/n\ZZ &\to& \SL_{n,\RR}\\
 1&\mapsto&
     \left(\begin{array}{cc}
      0 & -1 \\
      1 & \zeta_n+\zeta_n^{-1}
     \end{array}\right)
     \end{array}
     $$
     and the associated quotient singularity is an RDP
     of type $B_{\beta(n)}$ over $\RR$.
\end{enumerate}
\end{enumerate}
This can be proven along the lines of Proposition \ref{prop: fibers}
and we leave this to the reader.

\section{Allowing primes of bad reduction}
\label{sec: outlook}

In this section, we enhance the setup of the previous section by
allowing primes of bad reduction.
Let $\Sigma\subset\Spec\OO_K$ be a finite, possibly empty, subset of prime
ideals.
Then, $\Spec\OO_{K,\Sigma}=\Spec\OO_K\setminus\Sigma\subseteq\Spec\OO_K$ 
is a Zariski open affine subset, where $\OO_{K,\Sigma}$ denotes the
localization of $\OO_K$ away from $\Sigma$.
If $\Sigma=\emptyset$, then we recover the setup of the previous section.

\begin{Example}
Let $d\in\ZZ\setminus\{0,1\}$ be a squarefree integer.
Let $K=\QQ$, let $ L=\QQ(\sqrt{d})$, let $\OO_L\subset L$ be the ring of integers, 
and let $D\in\ZZ$ be the discriminant of $K\subset L$.
We have $D=d$ with $\OO_L=\ZZ[\sqrt{d}]$ if $d\not\equiv1\mod4$
or $D=4d$ with $\ZZ[\frac{1}{2}(1+\sqrt{d})]$ if $d\equiv 1\mod4$.
Then, 
$$
f\quad:\quad\Spec\OO_L\quad\to\quad\Spec\OO_\QQ\,=\,\Spec\ZZ
$$
is a finite flat morphism of degree 2.
It is \'etale of degree $2$ over the Zariski open set 
$\Spec\OO_{\QQ,\Sigma}=\Spec\ZZ\setminus\Sigma$, 
where $\Sigma$ is the finite set of primes dividing $D$.
We obtain a $\ZZ/2\ZZ$-torsor over $\Spec\OO_{\QQ,\Sigma}$
and for every $n\geq3$ a closed subgroup scheme
$$
\bmu_{n,\OO_{\QQ,\Sigma}}\wedge f^{-1}(\Spec\OO_{\QQ,\Sigma})
\quad\subset\quad \SL_{2,\OO_{\QQ,\Sigma}},
$$
which is fppf locally conjugate to the standard embedding
$\bmu_{n,\OO_{\QQ,\Sigma}}\subset\SL_{2,\OO_{\QQ,\Sigma}}$.
Since the subgroup scheme 
is not isomorphic to $\bmu_{n,\OO_{K,\Sigma}}$,
its embedding into $\SL_{2,\OO_{K,\Sigma}}$
is not $\GL_2(\OO_{K,\Sigma})$-conjugate to 
the standard embedding 
$\bmu_{n,\OO_{K,\Sigma}}\subset\SL_{2,\OO_{K,\Sigma}}$.
By Theorem \ref{thm: finite},
$\Klein(n,\ZZ)$ is a singleton and thus, the above
subgroup scheme of $\SL_{2,\OO_{K,\Sigma}}$ 
cannot be extended to a finite, flat, and linearly reductive
subgroup scheme of $\SL_{2,\ZZ}$.
\end{Example}

\subsection{At least one prime over $\mathbf{(2)}$ with good reduction}
We have the following straight forward 
generalization of Theorem \ref{thm: finite}.

\begin{Theorem}
 Let $K$ be a number field with ring of integers $\OO_K$.
 Let $\Sigma\subset\Spec\OO_K$ be a finite subset 
 that does not contain all primes lying over $(2)\in\Spec\ZZ$
 and let $\OO_{K,\Sigma}$ be the localization away from $\Sigma$.
 \begin{enumerate}
     \item For each integer $n\geq2$, the set $\Klein(n,\OO_{K,\Sigma})$ 
     consists of $\GL_2(\OO_{K,\Sigma})$-conjugacy classes of 
     subgroup schemes $G\subset\SL_{2,\OO_{K,\Sigma}}$
     that are locally in the fppf topology conjugate to
     the standard
     embedding $\bmu_{n,\OO_{K,\Sigma}}\subset\SL_{2,\OO_{K,\Sigma}}$.
     \item $\Klein(2,\OO_{K,\Sigma})$ is a singleton consisting
     of the $\GL_{2}(\OO_{K,\Sigma})$-conjugacy class of the standard
     embedding $\bmu_{2,\OO_{K,\Sigma}}\subset\SL_{2,\OO_{K,\Sigma}}$.
     \item If $n\geq3$, then $\Klein(n,\OO_{K,\Sigma})$ is a finite
     set and its cardinality is independent of $n$.
 \end{enumerate}
\end{Theorem}

\begin{proof}
(1) follows from Theorem \ref{thm: key}.

(2) follows from (1) and Corollary \ref{cor: mu2 case}.

(3) Independence of $n$ for $n\geq3$ follows from Corollary \ref{cor: independence of n}.
Finiteness can be shown along the lines of the proof of 
Theorem \ref{thm: finite}.(2) and follows from the finiteness of
$\Hfl{1}(\OO_{K,\Sigma},\Norm_{\GL_2}(\bmu_n))$, for which we use
the short exact sequence \eqref{eq: sequence again}:

A non-trivial class in $\Hfl{1}(\OO_{K,\Sigma},\ZZ/2\ZZ)$
can be represented by a quadratic field extension $K\subset L$
that is ramified at most in $\Sigma$.
By definition, we have $\QQ\subset K$, which is of some degree
$d:=[K:\QQ]$ and ramified in some finite subset $\Sigma'\subset\Spec\ZZ$.
Enlarging $\Sigma'$ by the finite set of primes lying under $\Sigma\subset\Spec\OO_K$,
we find that $\QQ\subset L$ is an extension of degree $2d$ that is
ramified at most in $\Sigma'$.
The set of such extensions of $\QQ$ is finite by Hermite's theorem, see, 
for example, \cite[Chapter V,\S4, Corollary to Theorem 5]{Lang}.
This implies finiteness of $\Hfl{1}(\OO_{K,\Sigma},\ZZ/2\ZZ)$.
(Alternatively, this can also be shown via finiteness of homomorphisms
$\Cl_K^\idealm\to\ZZ/2\ZZ$, where $\Cl_K^\idealm$ is the ray class group
for a suitable modulus $\idealm$ that depends on $\Sigma$.)

The fiber over some class $\eta\in\Hfl{1}(\OO_{K,\Sigma},\ZZ/2\ZZ)$ in
\eqref{eq: sequence again} is a quotient of 
$\Cl(\OO_{K',\Sigma'})^{\oplus2}$,
where $K'$ is equal to $K$ or to a quadratic extension of $K$.
Since such class groups are finite, so is the fiber over $\eta$.
\end{proof}

\subsection{No restriction on primes over $\mathbf{(2)}$}
The situation becomes considerably more difficult when trying to extend the results of
the previous section to $\Spec\OO_{K,\Sigma}$ in the case where $\Sigma$
contains all primes lying over $(2)\in\Spec\ZZ$.
Depending on $\Sigma$, 
the other group schemes of Theorem \ref{thm: Klein} will show up.
Thus, if $G\in\Klein(n,\OO_{K,\Sigma})$, then it may
happen that $G$ is not a twisted form of $\bmu_{n,\OO_K}$,
that is, Theorem \ref{thm: key} is not applicable.

For example, if $\Sigma$ contains all primes lying over $(2)$,
then next to the standard embedding $\bmu_{n,\OO_{K,\Sigma}}\subset\SL_{2,\OO_{K,\Sigma}}$
also group subschemes $G\subset\SL_{2,\OO_{K,\Sigma}}$ will show up
that are fppf locally conjugate to the
binary dihedral group scheme 
$\BD_{n,\OO_{K,\Sigma}}\subset\SL_{2,\OO_{K,\Sigma}}$,
which is a finite flat linearly reductive group scheme
of length $4n$ over $\OO_{K,\Sigma}$.
To describe such $G$, one has to compute
the normalizer  $\Norm_{\GL_2}(\BD_n)$ and one has to
understand its first cohomology, see Theorem \ref{thm: brian},
which we have done over fields in \cite{LS2}.
One will eventually run into the problem of describing
and controlling $\Hfl{1}(\Spec\OO_{K,\Sigma},H)$ 
for some finite and non-abelian group $H$ scheme,
which cannot be done so elegantly using class field theory
as in the $\bmu_n$-case.
This complicated matters a lot, but one can probably
still establish finiteness of $\Klein(n,\OO_{K,\Sigma})$
for every finite subset $\Sigma\subset\Spec\OO_K$ 
using Hermite's theorem.


\begin{thebibliography}{XXXXX}
\bibitem[AOV08]{AOV} D. Abramovich, M. Olsson, A. Vistoli, \emph{Tame stacks in positive characteristic},
Ann. Inst. Fourier (Grenoble) 58 (2008), no.4, 1057--1091.
\bibitem[Ar69]{ArtinApprox} M. Artin, \emph{Algebraic approximation of structures over 
complete local rings}, Inst. Hautes \'Etudes Sci. Publ. Math. No. 36 (1969), 23--58.
\bibitem[Ar77]{ArtinRDP} M. Artin, \emph{Coverings of the rational double points in characteristic $p$},
Complex analysis and algebraic geometry, pp. 11--22, Iwanami Shoten Publishers, Tokyo, 1977.
\bibitem[BLR90]{BLR} S. Bosch, W. L\"utkebohmert, M. Raynaud, \emph{N\'eron models},
Ergeb. Math. Grenzgeb. (3), 21, Springer (1990).
\bibitem[Br67]{Brieskorn} E. Brieskorn, \emph{Rationale Singularit\"ten komplexer Fl\"achen},
Invent. Math. 4 (1967/68), 336--358.
\bibitem[Ch92]{Chin} W. Chin, \emph{Crossed products of semisimple cocommutative Hopf
algebras}, Proc. Amer. Math. Soc. 116 (1992), no. 2, 321--327.
\bibitem[Co14]{Conrad} B. Conrad, \emph{Reductive group schemes},
Autour des sch\'emas en groupes. Vol. I, 93--444, Panor. Synth\`eses, 42/43, 
Soci\'et\'e Math\'ematique de France, 2014.
\bibitem[Co]{KConrad} K. Conrad, \emph{A non-free relative integral extension},
notes available from \url{https://kconrad.math.uconn.edu/blurbs/gradnumthy/notfree.pdf}
\bibitem[CR62]{CR} C.W. Curtis,  I. Reiner, 
\emph{Representation theory of finite groups and associative algebras},
Reprint of the 1962 original, AMS Chelsea Publishing, 2006.
\bibitem[Du79]{Durfee} A.H. Durfee, \emph{Fifteen characterizations 
of rational double points and simple critical points}, 
Enseign. Math. (2) 25 (1979), no. 1-2, 131--163.
\bibitem[Fo85]{Fontaine} J.-M. Fontaine, \emph{Il n'y a pas de vari\'et\'e ab\'elienne sur Z},
Invent. Math. 81 (1985), no. 3, 515--538.
\bibitem[FT91]{FT} A. Fr\"ohlich, M.J. Taylor, \emph{Algebraic number theory},
Cambridge studies in advanced mathematics 27, Cambridge University Press (1991).
\bibitem[Gi71]{Giraud} J. Giraud, \emph{Cohomologie non ab\'elienne}
Die Grundlehren der mathematischen Wissenschaften 179, Springer (1971).
\bibitem[Ha15]{Hashimoto} M. Hashimoto, \emph{Classification of the linearly reductive 
finite subgroup schemes of $\mathrm{SL}_2$},
Acta Math. Vietnam. 40 (2015), no.3, 527--534.
\bibitem[KLO25]{KLO} G. Kemper, C. Liedtke, C. Ott, 
\emph{On Noether's Degree Bound for Finite Group Schemes},
arXiv:2505.24752 (2025).
\bibitem[Kl84]{Klein} F. Klein, \emph{Vorlesungen \"uber das Ikosaeder und die Aufl\"osung der Gleichungen vom f\"unften Grade},
Reprint of the 1884 original. Edited, with an introduction and commentary by Peter Slodowy.
Birkh\"auser Verlag (1993).
\bibitem[La94]{Lang} S. Lang, \emph{Algebraic Number Theory}, 
Second edition, Grad. Texts in Math. 110, Springer (1994).
\bibitem[Li24]{LieMcKay} C. Liedtke, \emph{A McKay Correspondence in Positive Characteristic}, 
Forum Math. Sigma 12 (2024), Paper No. e116, 36 pp.
\bibitem[LMM]{LMM} C. Liedtke, G. Martin, Y. Matsumoto, \emph{Linearly Reductive Quotient Singularities},
arXiv:2102.01067 (2021), to appear in Ast\'erisque.
\bibitem[LS14]{LS} C. Liedtke, M. Satriano, \emph{On the birational nature of lifting}, Adv. Math. 254 (2014), 118--137.
\bibitem[LS25]{LS2} C. Liedtke, M. Satriano, \emph{On rational double points over nonclosed fields}, arXiv:2503.19787 (2025).
\bibitem[Lip69]{Lipman} J. Lipman, \emph{Rational singularities, with applications to algebraic 
surfaces and unique factorization}, 
Inst. Hautes \'Etudes Sci. Publ. Math. No. 36, (1969), 195--279.
\bibitem[Mi80]{Milne} J.S. Milne, \emph{\'Etale Cohomology}, 
PMS-33, Princeton University Press (1980).
\bibitem[Na61]{Nagata} M. Nagata, \emph{Complete reducibility of rational 
representations of a matric group}, J. Math. Kyoto Univ. 1 1961/1962, 87--99.
\bibitem[Ne99]{Neukirch} J. Neukirch, \emph{Algebraic number theory},
Grundlehren Math. Wiss. 322, Springer (1999).
\bibitem[No16]{Noether} E. Noether, \emph{Der Endlichkeitssatz der Invarianten endlicher 
Gruppen}, Math. Ann. 77 (1916), 89--92.
\bibitem[OT70]{OortTate} F. Oort, J. Tate, \emph{Group schemes of prime order}, 
Ann. Sci. \'Ecole Norm. Sup. (4) 3 (1970), 1--21. 
\bibitem[Sa12]{SatrianoCST} M. Satriano, \emph{The Chevalley--Shephard--Todd theorem 
for finite linearly reductive group schemes},
Algebra Number Theory 6 (2012), no. 1, 1--26.
\bibitem[Se97]{SerreGaloisCohomology} J.-P. Serre, \emph{Galois Cohomology}, 
Corrected reprint of the 1997 English edition, Springer Monogr. Math., Springer, 2002
\bibitem[St12]{Steinitz} E. Steinitz, \emph{Rechteckige Systeme und Moduln in algebraischen 
Zahlk\"orpern. II}, Math. Ann. 72 (1912), no. 3, 297--345.
\bibitem[Wo11]{Wolf} J. A. Wolf, \emph{Spaces of constant curvature}, Sixth edition.
AMS Chelsea Publishing (2011). 
\end{thebibliography}
\end{document}